\documentclass[10pt,leqno]{amsart}
\usepackage{graphicx}
\usepackage{indentfirst,csquotes}

\usepackage[english]{babel}
\usepackage{fancyhdr} 
\usepackage{subfig}

\usepackage{hyperref}
\usepackage{tikz}
\usepackage{dsfont}
\usetikzlibrary{calc}
\usepackage{amssymb}
\usepackage{mathtools}
\usepackage{amsmath}
\usepackage{latexsym}
\usepackage{amsthm}
\usepackage{esint}

\usepackage{comment}

\usepackage{etoolbox}
\usepackage{dirtytalk}
\usepackage[ruled]{algorithm2e}
\usepackage{longtable}
\usepackage{hyperref}
\hypersetup{
    colorlinks=true,
    linkcolor=blue,
    citecolor=teal
    }

\newcommand{\R}[0]{\mathbb{R}}
\newcommand{\N}[0]{\mathbb{N}}
\newcommand{\p}{\partial}
\newcommand{\dd}{\, \mathrm{d}}
\newcommand{\interior}{\, \mathrm{int}}
\newcommand{\Hau}{\mathcal{H}}

\newcommand{\dist}{\mathrm{dist}}
\newcommand{\delomega}{\partial\Omega}

\newcommand{\dH}{\, \mathrm{d} \mathcal{H}^{n-1}}
\newcommand{\dHo}{\, \mathrm{d} \mathcal{H}^{1}}

\newcommand{\one}{\mathds{1}}

\newtheorem{Conjecture}{Conjecture}
\newtheorem{Theorem}{Theorem}
\numberwithin{Theorem}{section}

\newtheorem{Lemma}[Theorem]{Lemma}
\newtheorem{Corollary}[Theorem]{Corollary}
\theoremstyle{definition}
\newtheorem{Definition}[Theorem]{Definition}
\theoremstyle{remark}
\newtheorem{Remark}[Theorem]{Remark}
\numberwithin{equation}{section}
\numberwithin{equation}{section}
 \DeclareMathOperator*{\Span}{span}

\newcommand{\gyula}[1]{{\color{blue}Gyula: {#1}}}

\makeatletter
\newenvironment{customproof}[1]{%
  \par\pushQED{\qed}%
  \normalfont\topsep6pt \trivlist
  \item[\hskip\labelsep\itshape
    Proof of #1\@addpunct{.}]\ignorespaces
}{%
  \popQED\endtrivlist\@endpefalse
}
\makeatother

\begin{document}
\title[Weighted perimeter of convex sets]{Weighted Perimeters and  pth Moments of Inertia of Convex Curves and Surfaces} 
\author[]{Gyula Csató, Davide Giovagnoli and Prosenjit Roy}
\date{\today}
\address{Gyula Csató: Facultat de Matemàtiques i Informàtica, Universitat de Barcelona\\ 
Gran Via de les Corts Catalanes, 585, 08007, Barcelona, Spain
}
\address{Centre de Recerca Matemàtica, Edifici C, Campus Bellaterra, 08193 Bellaterrra, Spain}
\email{gyula.csato@ub.edu}
\address{Davide Giovagnoli: Dipartimento di Matematica, Università di Bologna, Piazza di Porta \\ S.Donato 5, 40126, Bologna-Italy}
\email{d.giovagnoli@unibo.it}
\address{Prosenjit Roy: Department of Mathematics and Statistics, IIT-Kanpur, India}
\email{prosenjit@iitk.ac.in}
\maketitle

\begingroup
\let\thefootnote\relax
\footnotetext{\small  Fundings: D.G. is grateful to the Facultat de Matematiques i Informatica, Universitat de Barcelona for the warm hospitality. 
G.C. was supported by the Spanish AEIs projects PID2024-156429NB-I00, PID2021-125021NA-I00, and PID2021-123903NB-I00. } 

\endgroup

\tableofcontents 
\begin{abstract}

We study a shape optimization problem among convex bodies in $\mathbb{R}^n$ that minimize or maximize weighted perimeters of the form $\int_{\partial\Omega} \phi(|x|) \dd \mathcal{H}^{n-1}(x)$ under a standard perimeter constraint. We prove the existence of extremals for general weight functions in any dimension. In dimension two, we prove that the degenerate \textit{needle} configuration $(-a,a)\times \{0\} \subset \R^2$ is the optimizer for a wide family of weights, including $|x|^p$ for $p \in (0,2]$ and $|x|^{-\alpha}$ for $\alpha \in (0,1)$, among convex curves satisfying a symmetry assumption.

\end{abstract} 

\section{Introduction}

It is well known that the celebrated Fourier-analytic proof of the planar isoperimetric inequality is due to Hurwitz (1901). Less widely known is that, using essentially the same method, Hurwitz proved in 1902 \cite[pp.~396--397]{Hurwitz1902} the following inequality. Let $\Gamma\subset\R^2$ be a simple closed curve satisfying \begin{equation*} \int_{\Gamma} x \,\dHo(x) = (0,0), \end{equation*} and let $B_r$ denote the ball of radius $r$ centered at the origin such that $\partial B_r$ has the same length as $\Gamma$. Then \begin{equation}\label{gy:eq:intro Hurwitz seminal} \int_{\Gamma} |x|^2 \,\dHo(x) \leq \int_{\partial B_r} |x|^2 \,\dHo(x), \end{equation} with equality if and only if $\Gamma=\partial B_r$. In other words, among all closed curves of a given length whose barycenter is at the origin, the centered circle maximizes the weighted perimeter with weight $|x|^2$. This quantity is also known as the \emph{(angular) moment of inertia} of the curve. Hurwitz's inequality \eqref{gy:eq:intro Hurwitz seminal} appears to be one of the earliest variational results in which an extremal problem for a weighted perimeter is studied under a perimeter constraint. 

Several recent extensions of Hurwitz's theorem have been obtained. In particular, \cite{Charlotte2024} extends the inequality to inner parallel curves associated with a planar set $\Omega\subset\R^2$. These curves need not be simple or closed, even when $\partial\Omega$ is. The same work also considers higher-order moments of inertia, replacing the weight $|x|^2$ by $|x|^p$ for arbitrary $p>0$. On the other hand, \cite{LaMannaSannipoli} establishes a higher-dimensional analogue of Hurwitz's inequality and proves a quantitative stability version of the resulting inequality.

Another direction building on Hurwitz's inequality is to minimize the weighted perimeter among an appropriate class of curves. The two most notable results are by Sachs \cite{SachsIandII} and Hall \cite{Hall1985}, who realized that it also makes sense to minimize the moment of inertia if  one restricts the class of competing curves to the convex ones. In this case, minimizing  corresponds to the following physical question: take a wire of a given length and constant mass per unit length, and bend it into a closed convex loop in such a way that the loop has the lowest resistance to rotation when rotated around the origin.

The first of these notable results is that the equilateral triangle centered at the origin is the minimizer of the moment of inertia among convex curves in the plane. This was discovered by Sachs \cite{SachsIandII} and later a different proof was found by Hall \cite{Hall1985}. This is a surprising symmetry breaking result. Both proofs heavily rely on the specific value $p=2$ in the power of the weight function $|x|^p$: Sachs exploits that $|x|^2$ can be explicitly integrated along straight line segments, whereas Hall uses Fourier series methods. 

The second result is in Sachs \cite{SachsIandII}, who considered as well the same minimization question but restricted to  \textit{centrally symmetric} convex curves $\Gamma$, which are those that satisfy $x\in\Gamma$ if and only if $-x\in\Gamma.$ He proved that in this case
the degenerate curve,  the \textit{needle} $(-a,a)\times\{0\}$, (where the moment of inertia is twice the integral of $|x|^2$ along the interval) is one of the minimizers. He also proved that all minimizers are centrally symmetric parallelograms, understanding the needle as a degenerate rectangle.

After the works of Sachs and Hall, this problem seems to have been forgotten for a long time. Apparently, the only attempt to pick it up again was by Freitas, Laugesen and Liddell
\cite{FreitasLaugesenLiddell}, who considered for the first time the question of minimizing the angular moment of inertia of a convex surface embedded in $\R^3$ with a surface area constraint, that is, the $3$-dimensional analogue of the  results of Sachs and Hall in the plane. Their main results include the proof of the existence of a minimizer, as well as establishing the minimizers in 2 classes of polytopes (triangle prisms and rectangular prisms) and among the ellipsoids. The problem of finding the global minimizing convex surface remains open.

The results of this paper could be put in two categories. The first one concerns the existence of extremals for a broad class of weighted perimeter functionals in any dimension. Moreover we give some qualitative properties of these extremals in some special cases of weights. The second category consists of first attempts to  determine the shape of the extremals in dimension 2 when the weights are powers of the distance to the origin, or more generally, weights satisfying certain convexity and monotonicity assumptions. To start with, we  considered convex curves which have two orthogonal axes of symmetry meeting at the origin, which is less general than centrally symmetric convex curves.

Although we will state theorems involving quite general weighted perimeter functionals, the main examples that we have in mind are the infimum and supremum of
\begin{align}
  \label{eq:gy:def Ep for open}
  E_p(\partial\Omega) &:= \int_{\partial\Omega}|x|^p \dd \Hau^{n-1}(x) \quad \text{ if } p\geq 0, 
  \\
  \label{eq:gy:def Ealpha for open}
  E_{-\alpha}(\partial\Omega) &:= \int_{\partial\Omega}\frac{1}{|x|^{\alpha}} \dd \Hau^{n-1}(x) \quad \text{ if } 0\leq \alpha<n-1,
\end{align}
among all open convex sets $\Omega\subset\R^n$ with a perimeter constraint $|\partial\Omega|=C$. Here and throughout the paper $|A|=\Hau^{n-1}(A)$ denotes the $(n-1)$-dimensional Hausdorff measure of a set $A\subset\R^n$, and $C>0$ is some given constant. To solve this problem the value of $C$ is irrelevant due to the behavior of the functionals under scaling of $\Omega$. So in some theorems we will just take $C=1$ or some other convenient value depending on the setting.

Note that writing out both definitions \eqref{eq:gy:def Ep for open} and \eqref{eq:gy:def Ealpha for open} is redundant, but this will be convenient to always distinguish quickly whether we are in the case of the perimeter functional with singular weight $E_{-\alpha}$ or the other case $E_p$, as we will either minimize or maximize these energies depending on the case. Throughout this paper we will usually assume that $p\geq 0$ and $0\leq\alpha<(n-1)$, except in some proofs appearing in Section \ref{Sec:non_degeneracy}, where we allow  $p\geq -(n-1)$ to describe the full range of admissible parameters.  The other exception is the definition for degenerate sets in \eqref{eq:def degenerate Ep} below.

In the first case, i.e. for $E_p$ we are interested in the infimum $\inf E_p$, since the supremum is $+\infty$. In the second case $E_{-\alpha}$ we are interested in the supremum $\sup E_{-\alpha}$, since the infimum is $0$. These two facts follow by taking a sequence of  convex sets of a given perimeter and finite diameter, and moving them infinitely far away from the origin.

Our interest in $E_{-\alpha}$ stems from Csat\'o and Roy \cite{Csato2024} where apparently  for the first time $\sup E_{-\alpha}$ has been considered and it was proven that this supremum is finite. However, the method of proof in  \cite{Csato2024} does not establish the existence of a maximizer, nor can it provide information on the shape of a maximizer; not even under additional symmetry assumptions in dimension $2$.
The bound on the supremum of $E_{-\alpha}$ in \cite{Csato2024} was used to obtain a fractional Hardy inequality on convex hypersurfaces, respectively some generalizations of it.

If one removes the convexity assumption, then $\inf E_p=0$ and $\sup E_{-\alpha}=+\infty.$ To see this it is sufficient to take a sequence of manifolds  $\partial\Omega$ with $|\partial\Omega|=1$, which \say{oscillate} more and more near the origin, in the sense that  $\partial\Omega$ tends to have infinitely many leaves or tentacles which all pass close to the origin, or approaches a set which is dense in a neighborhood of the origin.

To state our first theorem on existence of extremals we need to enlarge the class of competitors  to \textit{convex bodies}, which are bounded closed convex sets, and extend the definitions given in \eqref{eq:gy:def Ep for open} and \eqref{eq:gy:def Ealpha for open}. Firstly, if a convex body $\Omega$ has nonempty interior, then  $E_p(\delomega)$ and $E_{-\alpha}(\delomega)$ are defined exactly as in \eqref{eq:gy:def Ep for open} and \eqref{eq:gy:def Ealpha for open} for an open convex set. However, if the convex body $\Omega$ is \textit{degenerate}, meaning that it has empty interior, or equivalently $\delomega=\Omega$ (here $\delomega$ denotes the topological boundary of a set $\Omega\subset\R^n$), then we define
\begin{align}\label{eq:def degenerate Ep}
  E_p(\partial\Omega)&:=2\int_{\partial\Omega}|x|^p \dd \Hau^{n-1}(x) \quad \text{if $\Omega$ is degenerate and $p>1-n.$}  
\end{align} 

It is easy to see that a convex body $\Omega\subset\R^n$ is degenerate if and only if it is contained in an $(n-1)$-dimensional plane.
Note that the only difference compared to the definitions \eqref{eq:gy:def Ep for open} and \eqref{eq:gy:def Ealpha for open} is the factor $2$ in front of the integral. This definition will ensure the continuity of the functional under convergence of convex bodies in the  Hausdorff distance. In view of the next theorem, keep in mind that for a convex body
$$
  E_0(\partial\Omega)=|\partial\Omega|\quad\text{ if $\Omega$ is nondegenerate}.
$$
We can now state our first theorem which generalizes and improves both the results of Freitas, Laugesen, and Liddell
 \cite{FreitasLaugesenLiddell} and that of Csat\'o and Roy \cite{Csato2024}. The result of \cite[Theorem 1]{FreitasLaugesenLiddell} only deals with $E_2$ in dimension $n=3$, whereas \cite[Theorem 1.2]{Csato2024} only establishes the inequality \eqref{eq:gy:universal constant p neg}, but not the existence of a maximizer. 
We will state our theorem for a large class of weight functions, which we introduce in the following definition.
 
\begin{Definition}[Weight function and weighted perimeter]\label{intro:gy:weight function} A \textit{weight function} $\phi$  satisfies the following assumptions.  $\phi$ is absolutely continuous on $(0,\infty)$ and $x' \in \R^{n-1} \mapsto \phi(|x'|) $ is integrable near the origin in $\R^{n-1}$, that is,
\begin{equation}\label{eq:integrability_phi}
\int_0^{1} |\phi(t)| t^{n-2} \dd t < \infty.
\end{equation}
Moreover, define for a convex body $\Omega\subset \R^n$
$$
  E_{\phi}(\delomega):=\eta_{\Omega}\int_{\p\Omega}\phi(|x|)
  \dd \Hau^{n-1}(x)\quad\text{where}\quad
  \eta_{\Omega}= \left\{\begin{array}{rl}
    1 & \text{ if $\Omega$ has nonempty interior,}
    \\
    2 & \text{ if $\Omega$ degenerate.}
  \end{array}\right.
$$
\end{Definition}

Further assumptions on the weights are specified when necessary.
Throughout this paper, we say that a function $\phi$ is increasing if $\phi(t) \leq \phi(s)$ whenever $t < s$, and strictly increasing if $\phi(t) < \phi(s)$ whenever $t < s$. We make the analogous convention for decreasing and strictly decreasing.

\begin{Theorem}
\label{Thm:intro:existence1}
Let $n\geq 2$.
\begin{itemize}
    \item[(i)] Let $\phi$ be an increasing weight function, then for every $\lambda>0$ there  exists a minimizer of
\begin{equation} \label{min_problem_p_pos}
   E_\phi(\p \Omega)\quad\text{among all convex bodies $\Omega\subset\R^n$ with $E_0( \p \Omega)=\lambda$.}
\end{equation}
In particular, if $\phi(t)=t^{p}$ for $p\geq 0$, there exists a constant $c_{n,p}$ depending only on $n$ and $p$ such that
\begin{equation}
  \label{eq:gy:universal constant p pos}
   E_{p}(\delomega)= \int_{\delomega}|x|^p\dH(x)\geq c_{n,p} |\partial\Omega|^{1+\frac{p}{n-1}},
\end{equation}
for all open bounded convex bodies $\Omega\subset\R^n$.

Moreover, if  $p> 4-n$ then every minimizer $\Omega^*$ is nondegenerate, that is $\Omega^*$ has to be the closure of a bounded open convex set.

\bigskip

\item[(ii)] Let $\phi$ be a decreasing weight function, then for every $\lambda>0$ there  exists a maximizer of
\begin{equation} \label{max_problem_p_neg}
   E_\phi(\p \Omega)\quad\text{among all convex bodies $\Omega\subset\R^n$ with $E_0( \p \Omega)=\lambda$.}
\end{equation}
In particular, if $\phi(t)=t^{-\alpha}$ for $0\leq\alpha<(n-1)$, there exists a constant $C_{n,\alpha}$ depending only on $n$ and $\alpha$ such that
\begin{equation}
  \label{eq:gy:universal constant p neg}
  E_{-\alpha}(\delomega)=\int_{\delomega}|x|^{-\alpha}\dH(x)\leq C_{n,\alpha} |\partial\Omega|^{1-\frac{\alpha}{n-1}},
\end{equation}
for all open bounded convex bodies $\Omega\subset\R^n$.

Moreover,  if $n\geq 5$ and $\alpha < n-4$ then every maximizer  $\Omega^*$ is nondegenerate, that is $\Omega^*$ has to be the closure of a bounded open convex set.
\end{itemize}
\end{Theorem}

One of the main difficulty in determining the possible candidates for the extremals of $E_p$ and $E_{-\alpha}$ is the lack of the applicability of the standard methods of the calculus of variations. This is due to the simultaneous constraints on convexity and perimeter. However, one simple variation that preserves both constraints is that of shifting the domain. Despite its simplicity, when calculating the second variation appropriately, this type of variation gives valuable information on the \say{location} of the maximizer of $E_{-\alpha}$. Note that the next theorem makes sense thanks to the previous one which establishes the existence of extremals.

\begin{Theorem} \label{thm:shifting_generalized}
    Let $\phi \in C^2(0, \infty)$ be a decreasing weight (resp. an increasing weight) and let $f(x):=\phi(|x|)$. If $\Delta f>0$ in $\R^n \setminus \{ 0\}$ and $\Omega \subset \R^n$ is a maximizer (resp. a minimizer) of $E_{\phi}$ among all convex bodies of equal perimeter $E_0( \p \Omega)=\lambda$ for some $\lambda>0$, then $0 \in \p \Omega$.
    
    In particular, the above result applies to $E_{-\alpha}$ whenever $n-2< \alpha < n-1$.
\end{Theorem}

Note that in the case $n=2$ for $E_{-\alpha}$, this result gives the full range of the admissible values of $\alpha$.
Theorem \ref{thm:shifting_generalized}, together with the  results we present here below for $n=2$, in particular Theorem \ref{thm:intro:alpha and n is 1}, suggest the following conjecture.

\begin{Conjecture}
\label{conjecture:intro:alpha}
For $n-2<\alpha<n-1$  the maximizer of $E_{-\alpha}$ in $\R^n$ is degenerate. More precisely it is the flat ball $\Omega=\{(x',0):\, x'\in\R^{n-1},\, |x'|\leq R\}$ for some radius chosen such that the given perimeter constraint $2\mathcal{H}^{n-1}(\Omega)=1$ is satisfied. 
\end{Conjecture}

For $n=2$, we have also performed a large number of numerical examples to support our intuition.
The second sentence of the conjecture is straightforward, if we assume that the maximizer is degenerate. It follows by a standard rearrangement argument in $\R^{n-1}.$

It is tempting to include in the conjecture any weight satisfying the hypotheses of Theorem \ref{thm:shifting_generalized} and not just $|x|^{-\alpha}$. But we lack numerical evidence for general weights $\phi$.

In Csat\'o and Roy \cite{Csato2024} it was implicitly conjectured that every $\alpha\in (0,n-1)$ would have led to degeneracy of the optimal set.  But here we correct the conjecture to a range of \say{highly} singular weights depending on the dimension. What we suggested in \cite{Csato2024} to be the optimal constant for \eqref{eq:gy:universal constant p neg} involving $E_{-\alpha}$  cannot be true for $n\geq 5$, in view of Theorem \ref{Thm:intro:existence1} Part $(ii)$. For $n-4\leq \alpha\leq n-2$ the maximizers might still be degenerate, but we do not have any result that would support or discourage this, not even numerical examples.

The second main theorem of this paper is a first partial result in dimension $2$ in support of this conjecture. It answers the conjecture affirmatively under a symmetry assumption on the competing curves $\Gamma$ ($=\delomega$), which is that they have two orthogonal axes of symmetry meeting at the origin: whenever $(x,y)\in\Gamma,$ then all four points $(\pm x,\pm y)$ are in $\Gamma$.  The theorem easily generalizes to slightly more general curves with less symmetry, which we put as Theorem \ref{thm:four quadrant convex}, but this generalization does not include centrally symmetric curves. However, for better comprehension, we first prefer to state our main results in less generality,  under this more common definition of symmetry. 
Being $E_\phi$ invariant under rotations around the origin, any result concerning the shape of the optimizers is valid up to rotations, obviously.

\begin{Theorem} \label{Thm:general_phi}
    Let $n=2$ and $\Gamma$ be  convex curve in $\R^2$, which has two orthogonal axes of symmetry meeting at the origin.  The following inequalities hold.
    \begin{itemize}
        \item[(I)]  For every $F : (0,\infty) \to \R$ convex and decreasing, such that  $t\mapsto F(t^2)$ is in $L^1((0,1)),$ the following inequality holds
        \begin{equation} \label{eq:ineq_generic_phi_convex_dec}
       E_{F(|\cdot|^2)}(\Gamma) =\int_{\Gamma} F(|x|^2) \dHo(x) \leq 4 \int_0^{|\Gamma|/4} F(t^2) \dd t.
        \end{equation}
        \item[(II)] For every $F : (0,\infty) \to \R$ concave and increasing, such that  $t\mapsto F(t^2)$ is in $L^1((0,1)),$ the following inequality holds
        \begin{equation} \label{eq:ineq_generic_phi_concav_increase}
            E_{F(|\cdot|^2)}(\Gamma) = \int_{\Gamma} F(|x|^2) \dHo(x) \geq 4 \int_0^{|\Gamma|/4} F(t^2) \dd t.
        \end{equation}
    \end{itemize}
    In both \eqref{eq:ineq_generic_phi_convex_dec} and \eqref{eq:ineq_generic_phi_concav_increase}, equality holds if $\Gamma$ is the needle.

    Moreover, if there is equality in \eqref{eq:ineq_generic_phi_convex_dec} and $F$ is strictly convex, respectively there is equality in \eqref{eq:ineq_generic_phi_concav_increase} and $F$ is strictly concave, then $\Gamma$ is the needle. 
    
   Furthermore if, in addition, $t \mapsto F(t^2)$ is either concave in Part $(I)$ or convex in Part $(II)$ we can remove the assumption that the two orthogonal axes of symmetry meet at the origin.
    \end{Theorem}

Note that the assumption $t\mapsto F(t^2)$ is in $L^1((0,1))$ is equivalent to being in $L^1((0,\lambda))$ for every $\lambda>0,$ due to the other assumptions on $F.$ 
    
As an immediate corollary of Theorem \ref{Thm:general_phi}, we obtain three interesting special cases. The first two results we put together in a Corollary, as they are both related to the work of Sachs \cite{SachsIandII}. Firstly, we extend a theorem of Sachs for $p=2$ to the range $p\in (0,2)$. Secondly, we mention the case of the logarithmic potential, as a similar energy has been considered by Sachs too, namely the double integral 
$
\int_{\Gamma}\int_{\Gamma}\log(|x-y|)\dHo(x)\dHo(y).
$
The third special case we state as a separate theorem, since we have found for that an independent and simpler proof, provided in Section \ref{section:singular}.

\begin{Corollary} \label{Cor:intro:p and log ineq}
    Let $n=2$ and $\Gamma$ be a convex curve in $\R^2$, which has two orthogonal axes of symmetry meeting at the origin.   Then, for $p \in (0,2]$, it holds that
    \begin{equation} \label{eq:p_inequality}
        E_p(\Gamma)=\int_{\Gamma} |x|^{p} \dHo(x) \geq \frac{1}{(p+1)4^p} |\Gamma|^{p+1}.
    \end{equation}
    Furthermore,
    \begin{equation} \label{eq:log_inequality}
        E_{\log}(\Gamma)= -\int_{\Gamma} \log(|x|) \dHo(x) \leq |\Gamma| \left( 1- \log \left(\frac{|\Gamma|}{4} \right)\right).
    \end{equation}
    If $p<2$ then equality holds in \eqref{eq:p_inequality} , respectively in \eqref{eq:log_inequality}, if and only if $\Gamma$ is the needle. 
    
    In addition, if $1 \leq p\leq 2$, we can remove the assumption that the two orthogonal axes of symmetry meet at the origin.
\end{Corollary}

The next theorem is also an immediate consequence of Theorem \ref{Thm:general_phi}, but we highlight it here for two reasons. Firstly, it is the best result to support Conjecture \ref{conjecture:intro:alpha}. Secondly, we have found an independent proof which is simpler and does not rely on the case of the weight $|x|^2$. More precisely, the proof of Theorem \ref{Thm:general_phi} consists of two main steps, of which the first one is by proving the special case $F(t)=t$, and which we have put in Section \ref{section:p=2}. This step is not needed if we are only interested in $E_{-\alpha}.$ Last but not least, the next theorem provides an alternative proof for \eqref{eq:log_inequality}; see Remark \ref{remark:Almut}.

\begin{Theorem}
\label{thm:intro:alpha and n is 1}
Let $0<\alpha<1$, $n=2$ and $\Gamma$ be a convex curve in $\R^2$, which has two orthogonal axes of symmetry meeting at the origin.  The following inequality holds
\begin{equation} \label{eq:intro:thm2 alpha}
E_{-\alpha}(\Gamma)=\int_{\Gamma} |x|^{-\alpha}  \dHo(x) \leq 
\frac{4^\alpha}{1-\alpha}|\Gamma|^{1-\alpha}.
\end{equation}
There is equality in  \eqref{eq:intro:thm2 alpha} if and only if $\Gamma$ is a needle.
\end{Theorem}

We have stated the above three results of Theorem \ref{Thm:general_phi}, Corollary \ref{Cor:intro:p and log ineq} and Theorem \ref{thm:intro:alpha and n is 1}, for convex curves in $\R^2$, which have two orthogonal axes of symmetry meeting at the origin, but they immediately extend to more general curves by a straightforward one dimensional symmetrization argument.

\begin{Definition}\label{def:4 quadrant convex}
We say that a set $\Gamma\subset\R^2$ is \textit{$4$-quadrant-convex}  if $\Gamma$ is the disjoint union (up to a set of measure zero) of four sets $\Gamma_1,\ldots,\Gamma_4$. For each $i=1,\ldots,4$ there exists a rotation around the origin $O_i$ such that
$$
  O_i(\Gamma_i)=S_i\cap \{(x,y)\in\R^2;\, x\geq 0,\, y\geq 0\}
$$
for some  convex curve $S_i$ that  has two orthogonal axes of symmetry meeting at the origin. The sets $S_i$ might be empty or contained in just one of the positive axes $\{x\geq 0,\, y=0\}$ or $\{x=0,\, y\geq 0\}.$
\end{Definition}

The main example for Definition \ref{def:4 quadrant convex} we have in mind is the boundary of $\{f(x) \leq y\leq g(x),\, x\in (-a,b)\}$,  provided $a,b\geq 0$, $f$ is convex, $g$ is concave, $f(-a)\leq 0\leq  g(-a)$, $f(b)\leq 0\leq g(b)$, and $f$ attains its minimum, respectively $g$ its maximum, at the origin. But there are sets which are not necessarily simply closed  and are $4$-quadrant-convex. An example of a convex curve which is not $4$-quadrant-convex is for instance a triangle with one of its corners lying on the positive $x$ axis, one on the negative $x$-axis, and the third corner not lying on the $y$-axis (and assuming the triangle is nondegenerate). 

With Definition \ref{def:4 quadrant convex} we are able to state our final result.

\begin{Theorem}\label{thm:four quadrant convex}
Theorem \ref{Thm:general_phi} remains true if $\Gamma$ is $4$-quadrant-convex.
\end{Theorem}

Let us make a brief comment on the proofs. Many of the methods we apply seem  new and several of them have been developed from scratch to tackle these kind of problems.   Even the proof of Theorem \ref{Thm:intro:existence1} when restricted to (i), $\phi(t)=t^2$ and $n=3$, is not the same one as the existence theorem provided by \cite{FreitasLaugesenLiddell}, because for general weights $\phi$ one cannot assume the monotonicity of $E_{\phi}$ under inclusion. However there are two exceptions, where we apply standard methods of the calculus of variations. Firstly, the proofs of the two statements in Theorem \ref{Thm:intro:existence1} showing that for certain ranges of values of $p$ and $\alpha$ the extremals are nondegenerate is based on a straightforward but  tedious calculation of the first and second variation of  cylinders. Finally, the proof of Theorem \ref{thm:shifting_generalized} is also a consequence of calculating the second variation of the family of competitors obtained by shifting a maximizer in all possible directions.

As perhaps a highlight of our new methods, we would like to point out the proof of Theorem \ref{Thm:general_phi}. To the best of our knowledge this is where for the first time majorization theory, more precisely a continuous version of Karamata's inequality, originally due to Hardy, Littlewood and Polya \cite{hardy1929some}, combined with a rearrangement argument, is used to solve such a shape optimization problem. Its power is that using a result on $E_2$, one can solve in one stroke the corresponding  minimization problems for  $E_p$ on the entire range $p\in (0,2)$, as well as the maximization problem for $E_{-\alpha}$ for all $\alpha\in (0,1).$ 

Last but not least, let us point out on the example of Theorem \ref{thm:intro:alpha and n is 1}, that what we are dealing with in this paper are sharp functional inequalities of the sort
\begin{equation}\label{intro:functional ineq}
  \int_0^L \frac{\sqrt{1+f'(x)^2}}{(x^2+f(x)^2)^{\alpha/2}}\dd x \leq \frac{1}{1-\alpha} \left(\int_0^L \sqrt{1+f'(x)^2} \dd x \right)^{1-\alpha}
\end{equation}
for every nonnegative, concave decreasing function $f$ on $[0,L]$, with equality if and only if $f=0.$ It is easy to see that all of  the three conditions (nonnegative, concave, decreasing) are necessary for this inequality to hold.

\subsection{Notation.}
If $A$ is a subset of $\R^n$, then $\p A$ denotes its topological boundary in $\R^n$. In particular if $A\subset\R^n$ is a degenerate convex set, then $\p A=A.$ 

Throughout the paper $B_r^k(x_0)$ is the $k$-dimensional ball in $\R^k$ of radius $r$ and center $x_0$. If $k=n,$ we omit the index $k$ and just write $B_r(x_0).$  Whenever the center is not specified we intend it to be the origin. We use the notation $\omega_{k}:=\mathcal{H}^{k}(B_1^k)$ and $\beta_{k-1}:=\mathcal{H}^{k-1}(\p B_1^k)$, recalling the relation $\beta_{k-1}= k \omega_k$. 

We recall for convenience the notion of Hausdorff distance $d_H$. 
For $\Omega_1,\Omega_2 \subset \R^n$
\[
d_H(\Omega_1,\Omega_2):= \inf \{\delta >0 \, : \, \Omega_1 \subset \Omega_2 + B_\delta, \, \Omega_2 \subset \Omega_1 + B_\delta \}
\]
where we use the notation of Minkowski sum between sets. 

We also recall that we call in this paper a function $f$ \textit{increasing} if $f(x)\leq f(y)$ for all $x<y$ and do not employ the commonly used word nondecreasing. If the inequality is strict, we will say so explicitly.

\smallskip
\section{Results in arbitrary dimensions}

\subsection{Existence of Extremals} \label{Sec:existence}

This section is devoted to proving the existence of extremals for the weighted energy functionals $E_\phi$ in Theorem \ref{Thm:intro:existence1}. The statement of the theorem on nondegeneracy will be proven in Section \ref{Sec:non_degeneracy}.

At one step of the proof we will need the continuity of the energy functionals $\Omega\mapsto E_\phi(\p\Omega)$ with respect to the convergence of a sequence of convex bodies $\{\Omega_i\}_{i\geq1}$ in the  Hausdorff metric (also called Hausdorff distance). This result is not very surprising but somewhat technical, so we have put it into a companion paper \cite{CG_part2}.  
It generalizes the well-known result on the continuity of the standard perimeter ($\phi\equiv1$) of convex bodies with respect to the Hausdorff metric, which can be found in \cite[Result 3.2.36, p. 272]{federer}; or see also \cite[Section 2.3]{Bucur} for another proof. In the case $\phi(t)=t^2$ and $n=3$ the continuity has already been commented upon in \cite{FreitasLaugesenLiddell}.

It is in Theorem \ref{lm:continuity_energy} that the assumption of absolute continuity in the definition of the weight function is used, as well as the factor $\eta_{\Omega}=2$ in Definition \ref{intro:gy:weight function} of $E_{\phi}$ in case that $\Omega$ is degenerate. 

\begin{Theorem}[Continuity of the energy functional]\label{lm:continuity_energy}
       Let $\phi$ be a weight function according to Definition \ref{intro:gy:weight function} and 
       $\{\Omega_i\}_{i \geq 1}$ be a sequence of bounded convex bodies. If $\Omega_i$ converges to $ \Omega$ for $i \to \infty$ in the Hausdorff metric then
       \[
       \lim_{i\to\infty} E_\phi(\p\Omega_i)= E_{\phi}( \p\Omega) .
       \]
\end{Theorem}

A fundamental dichotomy between increasing and decreasing weights lies in their monotonicity with respect to set inclusion. For instance, for increasing and nonnegative weights, the energy $E_\phi$ increases with inclusion, provided the inner set contains the origin. In contrast, this monotonicity generally fails for decreasing weights. Note that the requirement that the origin lies inside the inner set avoids any contradiction with the non-monotonicity result in \cite[Theorem 1.1]{saracco2024monotonicity}, where it has been proven that if  $E_\phi(\p A)\leq E_\phi(\p B)$ for every pair of convex bodies $A\subset B$, then $\phi$ must be constant. We state our monotonicity theorem here without proof, deferring it to our companion paper \cite{CG_part2}, as we will not need it in what follows. However, using monotonicity, when available, can provide a simpler proof for the existence of minimizers, as we shall explain after the theorem.

\begin{Theorem}[Monotonicity under inclusion for $E_\phi$] \label{l:monotonicity}
     Let $\phi$ be a  nonnegative and increasing function. If $\Omega_1,$ $\Omega_2$ are two convex bodies such that $0 \in \Omega_1 \subset \Omega_2$, then $E_\phi(\p \Omega_1) \le E_\phi(\p \Omega_2)$. 
\end{Theorem}

Nonnegativity is also a necessary assumption of the theorem, since monotonicity with respect to inclusion cannot hold if $\phi$ changes sign. To see this, consider, for example,
$\phi(t)=t-1$, which is obviously increasing in $t.$ However, $E_\phi(\p B_{1/4}) > E_\phi(\p B_{1/3})$ and $E_\phi(\p B_1) < E_\phi(\p B_2)$.

As remarked, such monotonicity cannot hold for decreasing weights. As an example, it is easy to find two convex sets $\Omega_1$ and $\Omega_2$ both containing the origin and contained in the ball $B_1\subset\R^n$ such that $E_{-\alpha}(\p \Omega_1)<E_{-\alpha}(\p B_1)$ but $E_{-\alpha}(\p \Omega_2)>E_{-\alpha}(\p B_1)$. More precisely, for every ball $B_{r}$ with $0<r<1$ it holds that $E_{-\alpha}(\p B_{r}) < E_{-\alpha}( \p B_{1})$. But if we take $0 \in B_{1}^{n-1} \times \{ 0\} =\{ (x',0) \, : \, |x'|\leq 1 \}\subset B_1 \subset \R^n$ then 
\[
E_{-\alpha}( B_{1}^{n-1} \times \{ 0\} ) > E_{-\alpha}( \p B_{1}) \quad \text{ for all $\alpha$ sufficiently close to $n-1$}.
\]

We now shift our focus to prove the existence of extremals via the direct method. For this we must guarantee the compactness of minimizing (or maximizing) sequences. In dimension $n=2$, this is trivial since the diameter of a convex set is universally bounded by its perimeter. However, for $n \ge 3$, a sequence of convex bodies can have constant perimeter while stretching out to infinity. To prevent this loss of compactness under the constraint $E_0(\p \Omega)=1$, in Lemma \ref{lm:bound_on_max_radius_decreasing} we establish a crucial bound relating the  energy to the maximal radius of the set.  

For increasing and nonnegative weights the monotonicity property of Theorem \ref{l:monotonicity} is available and the proof of compactness simplifies. We carry out this alternative proof in \cite{CG_part2}. 
The strategy in such a case is to generalize the result obtained in \cite[Lemma 8]{FreitasLaugesenLiddell} where the authors deal with the case of the weight $|x|^2$ in $\R^3$.
The idea is to bound the energy of the convex body $\Omega$ from below by comparing it to the energy of a suitably constructed simplex contained within it and exploit monotonicity.

In the absence of monotonicity of the energy with respect to inclusion due to either the negativity of the weight or due to a decreasing weight function, we bound the energy of the convex body $\Omega$ by the sum of the energies of its projections onto the $n$ canonical hyperplanes. 
The idea of considering projections onto planes was also employed in \cite{Csato2024} to prove that the energy $E_{-\alpha}$ is bounded. At the same time, we still need the above mentioned simplices contained in $\Omega$ to estimate the standard perimeter from below. 

We start with a preliminary technical lemma, before stating the main Lemma \ref{lm:bound_on_max_radius_decreasing}.

\begin{Lemma}\label{lm:bathtub_principle}
    Let $k \in \{1, 2, \dots, n\}$ and let $\phi: [0, \infty) \to \R$ be an increasing (resp. decreasing) function such that $z \mapsto \phi(|z|)$ is locally integrable in $\R^k$. Assume $v \in L^{\infty}(\R^k)$ has compact support, satisfies $0 \leq v \leq K$ a.e. in $\R^k$, and
    $$
    \int_{\R^k} v(z) \dd z = D,
    $$
    for some constants $K>0$ and $D\geq 0$. If $r\geq0$ is chosen such that the $k$-dimensional volume of the ball satisfies $\mathcal{H}^k(B_r^k) = \frac{D}{K}$, then
    \begin{equation}
        \int_{\R^k} \phi(|z|) v(z) \dd z \geq K \int_{B_r^k} \phi(|z|) \dd z \quad \text{(resp. $\leq$)}.
    \end{equation}
\end{Lemma} 

\begin{proof}
    The case $D=0$ is trivial, since then $v \equiv 0$ a.e. and $r=0$.  

    Let $w(z)=K\one_{B_r^k}(z)$. Since $\int_{\R^k} w(z) \dd z = D$, we have by assumption that
    \begin{equation}\label{eq:bathtub_zero_mass}
        \int_{\R^k} (v(z)-w(z)) \dd z = 0.
    \end{equation}
     Since $\phi$ is increasing, $v \leq w$ a.e. in $B_r^k$, and $v \geq w$ a.e. in $\R^k \setminus B_r^k$, we deduce that the function $z \mapsto \big(\phi(|z|) - \phi(r)\big) \big(v(z) - w(z)\big)$ is nonnegative a.e.  in $\R^k$.  Hence we have 
\begin{equation*}\label{eq:bathtub_positive_integral}
        \int_{\R^k} \big(\phi(|z|) - \phi(r)\big) \big(v(z) - w(z)\big) \dd z \geq 0.
    \end{equation*}
  Expanding the product and using \eqref{eq:bathtub_zero_mass} yields
    $$
    \int_{\R^k} \phi(|z|) \big(v(z) - w(z)\big) \dd z \geq 0,
    $$
    which gives the desired result.
    The decreasing case follows by applying the previous result to $-\phi$.
\end{proof}


We now state the result which will ensure the compactness of the maximizing/minimizing sequence. We use therein the common notation 
$$
  \phi^+=\max (0,\phi),\quad\quad \phi^-=\max(0,-\phi)\quad\text{ and }\quad \phi=\phi^+-\phi^-.
$$

\begin{Lemma} \label{lm:bound_on_max_radius_decreasing}
    Let $n\geq 3$ and let $\Omega$ be a convex body in $\mathbb{R}^n$ with $\mathcal{H}^{n-1}(\p \Omega) >0$ such that $0 \in {\Omega}$. Define $R:=\max \{|x| : x \in \p \Omega\}$. There exist positive constants $C_0, C_1, \gamma$ depending only on the dimension $n$, such that, defining the radius
    \begin{equation} \label{eq:def_radius_r}
        r := \left( C_0 \frac{E_0(\p \Omega)}{R} \right)^{\frac{1}{n-2}},
    \end{equation}
    the following statements hold.
    
  If $\phi$ is an increasing weight function  then 
    \begin{equation}\label{eq:bound_Ephi_by_R_inc}
        \frac{E_{\phi}(\p \Omega)}{E_0(\p \Omega)} \geq \frac{1}{\gamma R} \int_0^{\gamma R} \phi^+(t) \dd t -  \frac{C_1}{E_0(\p \Omega)} \int_{-R}^R \int_{B^{n-2}_{r}} \phi^-(|x'|) \dd x'.
    \end{equation}
If $\phi$ is a decreasing weight function, then 
    \begin{equation}\label{eq:bound_Ephi_by_R_dec}
      \frac{E_{\phi}(\p \Omega)}{E_0(\p \Omega)} \leq \frac{C_1}{E_{0}(\p \Omega)} \int_{-R}^R \int_{B^{n-2}_{r}} \phi^+(|x'|) \dd x' - \frac{1}{\gamma R} \int_0^{\gamma R} \phi^-(t) \dd t.
    \end{equation}
    
    Moreover, for any sequence of convex bodies $\{\Omega_i\}_{i\geq 1}$ with $E_0(\p \Omega_i)>0$ constant for all $i$ and $R_i \to \infty$, and for any increasing (resp. decreasing) weight $\phi$, it holds
    $$\lim_{i \to \infty} \frac{E_\phi( \p \Omega_i)}{E_0(\p \Omega_i)} = \sup_{t\geq 0} \phi(t), \quad \text{(resp.} \ =\inf_{t\geq 0} \phi(t)).$$
\end{Lemma}

Let us explain roughly the idea behind Lemma \ref{lm:bound_on_max_radius_decreasing} on the example of $n=3$ and $E_{-\alpha}.$ In that case, one easily sees that for the family of degenerate convex rectangles $\Omega_R=[0,R]\times[0,1/R]\times\{0\}\subset\R^3$ of constant perimeter $2\mathcal{H}^{2}( \Omega_R)=E_0(\p\Omega_R)=2$ it holds that
$$
  \lim_{R\to\infty}E_{-\alpha}(\p \Omega_R)=0=\inf_{t>0}t^{-\alpha}.
$$
This corresponds to the last statement of the lemma. Whereas the inequalities \eqref{eq:bound_Ephi_by_R_inc} and \eqref{eq:bound_Ephi_by_R_dec} are designed to estimate the energies of a general convex body by such a degenerate  rectangle, or more generally, by degenerate cylinders in dimensions $n>3$. 

Before providing the proof, we state a corollary showing how Lemma \ref{lm:bound_on_max_radius_decreasing} reads for the typical energies $E_p$ and $E_{-\alpha}$.

\begin{Corollary}\label{corollary:estimates Ep Ealpha and R}
Let $n\geq 3$, $p>0$ and $0<\alpha<n-1$, and suppose that $\Omega\subset\R^n$ is a convex body with $\mathcal{H}^{n-1}(\p\Omega)>0$ such that $0\in\Omega$. Define $R:=\max \{ |x| : x \in \p \Omega\}$. Then the following statements hold true.

\smallskip
(i) There exists a constant $C_{n,p}$ depending only on $n$ and $p$ such that
$$
E_p(\delomega)\geq C_{n,p} E_0(\delomega)R^p.
$$

(ii) There exists a constant $C_{n,\alpha}$ depending only on $n$ and $\alpha$ such that for all $R$ sufficiently bigger than some function of $E_0(\delomega)$, it holds that
$$
E_{-\alpha}(\delomega)\leq C_{n,\alpha}    \begin{cases}    E_0(\delomega) R^{-\alpha} & \text{ if } 0<\alpha<1, \\    E_0(\delomega) R^{-1}\log(R) & \text{ if } \alpha=1, \\ \displaystyle    \left( \frac{E_0(\delomega)}{R} \right)^{\frac{n-1-\alpha}{n-2}} & \text{ if } 1<\alpha<n-1.    \end{cases}
$$
\end{Corollary}
\begin{proof}
The proof of $(i)$ is an immediate application of \eqref{eq:bound_Ephi_by_R_inc}. Whereas the proof of $(ii)$ follows from \eqref{eq:bound_Ephi_by_R_dec} by means of the following calculation (we denote here $x'=(x_1,x'')\in\R\times\R^{n-2}$). 
$$
E_{-\alpha}(\delomega)\leq C\int_{-R}^R\int_{B_r^{n-2}}\big(x_1^2+|x''|^2\big)^{-\alpha/2}\dd x''\dd x_1 = C\int_0^r \beta_{n-3} s^{n-3} \int_{-R}^R \big(x_1^2+s^2\big)^{-\alpha/2} \dd x_1 \dd s.
$$
Substituting $x_1=st$ in the inner integral gives
$$
E_{-\alpha}(\delomega)\leq C \beta_{n-3}\int_0^r s^{n-2-\alpha} \left( \int_{-R/s}^{R/s} (1+t^2)^{-\alpha/2} \dd t \right)\dd s.
$$
On the one hand, by definition \eqref{eq:def_radius_r}  we have $r\to0$ as $R\to\infty$. And on the other hand $R/s \to \infty$ as $R \to \infty$ for all $s \in (0, r]$. We therefore split the analysis depending on $\alpha$.

\smallskip
If $1<\alpha<n-1$, we can estimate the integration in $t$ by the inclusion $(-R/s,R/s)\subset(-\infty,\infty)$, yielding 
$$E_{-\alpha}(\delomega)\leq C r^{n-1-\alpha} = C_{n,\alpha} (E_0(\delomega)/R)^{\frac{n-1-\alpha}{n-2}}.
$$

\smallskip
If $0<\alpha \le 1$, we split the inner integral into the intervals $(0,1)$ and $(1,R/s)$, bounding $(1+t^2)^{-\alpha/2} \leq t^{-\alpha}$ on the latter. This gives
\begin{equation} \label{eq:split_integral}
    \int_{0}^{R/s} (1+t^2)^{-\alpha/2} \dd t \leq 1 + \int_1^{R/s} t^{-\alpha} \dd t
\end{equation}
In view of \eqref{eq:def_radius_r} it holds  that $1\leq C(R/r)^{1-\alpha}\leq C(R/s)^{1-\alpha}$ for all $s\in(0,r)$ for $R$ sufficiently large compared to $E_0(\delomega)$.
Thus, 
for $0<\alpha<1$, the right-hand side of \eqref{eq:split_integral} is bounded by $C (R/s)^{1-\alpha}$ for $R$ sufficiently large $R$, for some constant $C$ depending only on $n$ and $\alpha.$  
Thus
$$
E_{-\alpha}(\delomega) \leq C R^{1-\alpha} \int_0^r s^{n-3} \dd s = C_{n,\alpha} R^{1-\alpha} r^{n-2} = C_{n,\alpha} E_0(\delomega) R^{-\alpha}.
$$

\smallskip
For $\alpha=1$, \eqref{eq:split_integral} is equal to $1 + \log(R/s)$. Thus we obtain that 
$$
E_{-1}(\delomega) \leq C (1+\log(R)) \int_0^r s^{n-3} \dd s -\int_0^r s^{n-3}\log(s) \dd s.
$$
Taking into account that a primitive of $s^{n-3}\log(s)$ is $s^{n-2}\log(s)/(n-2)-s^{n-2}/(n-2)^2$, one easily sees that in this case $E_{-1}(\delomega)$ can be bounded by $E_0(\p \Omega) R^{-1}\log(R),$ modulo some constant depending only on $n.$ 
\end{proof}

\smallskip

\begin{customproof}{Lemma \ref{lm:bound_on_max_radius_decreasing}}

We only need to deal with the case where $\phi$ is an increasing function. 
The case of a decreasing weight $\phi$ follows immediately by applying \eqref{eq:bound_Ephi_by_R_inc} to the increasing function $ -\phi$, instead of $\phi$. Multiplying both sides of the inequality by $-1$ flips the inequality and yields exactly \eqref{eq:bound_Ephi_by_R_dec}. The statement regarding a sequence of convex bodies with $R_i\to\infty$ also follows as  $\sup (-\phi) = -\inf \phi.$

    Let $\Omega$ be a convex body in $\R^n$ and let $y_1 \in \partial\Omega$ be a point at maximum distance from the origin, so that $|y_1|=R$. Let $V_1:= \Span \{y_1\}$. Define $y_2 \in \p \Omega$ be the point farthest from $V_1$. More precisely, consider
\[
b_2:= \max \{ \dist(V_1,z) \, : \, z \in \p \Omega \}>0,
\]
and $y_2$ is the point where the maximum is achieved. Denote with $V_2=\Span \{y_1,y_2\}$. Iteratively define
\[
b_j:=\max \{ \dist(V_{j-1},z) \, : \, z \in \p \Omega \}, \quad \text{for }j=2,\dots, n-1,\;\text{ and }\; b_1:=R
\]
and the corresponding points $y_j$ where such maximum is achieved.  Notice that $b_j>0$ since  $\mathcal{H}^{n-1}(\p \Omega) >0$.
Let us define for $j=1,\ldots,n-1$, the $j$-dimensional simplex $\Delta_j$ of vertices $0,y_1, \dots, y_j$
by
$$
  \Delta_j:=\left\{ \sum_{i=1}^{j}t_iy_i;\,t_i\geq 0\,, \sum_{i=1}^{j}t_i\leq 1\right\},\,\text{ and }\, \Delta:=\Delta_{n-1}.
$$
It is easy to see 
that $\Delta$ is contained in $\Omega$.

The $(n-1)$-dimensional area $\mathcal{H}^{n-1}(\Delta)$ of the simplex $\Delta$ is given by
\begin{equation}\label{eq:gy:area of simplex}
\mathcal{H}^{n-1}(\Delta)=\frac{1}{(n-1)!}b_1  \cdot \cdot \cdots \cdot b_{n-1},
\end{equation}
see for instance \cite[Formula 9.12.4.3]{berger2009geometry}.

Moreover, the vertices of $\Delta$ are $\{0,y_1,\dots,y_{n-1} \}$ thus, by calling $P$ the projection of $\Omega$ onto $V_{n-1}$, we infer that $\Delta \subset P(\Omega)$. Since $\Omega$ is convex it holds that $E_0(P(\Omega)) \leq E_0(\p \Omega)$. Combining these two facts with the monotonicity of the perimeter functional of convex sets with respect to inclusion, we must have that $E_0(\Delta)\leq E_0(P(\Omega))$, which in turn gives $E_0(\Delta) \leq E_0(\p \Omega)$, and thus
\begin{equation}
 \label{eq:b times R const}
  b_2 \cdot \cdots \cdot b_{n-1}\leq (n-1)!\, \frac{E_0(\partial\Omega)}{b_1}
  =(n-1)!\, \frac{E_0(\partial\Omega)}{R}.
\end{equation}

By definition of $R$ we have that, up to a rotation, $\Omega \subset [-R,R] \times \R^{n-1}$. Similarly by construction and a simple induction argument, it follows that $\Omega$ is contained, up to a rotation, in the hyperrectangle
\begin{equation}
  \label{eq:Omega in cylinder}
\Omega\subset C_{R,b}:=[-R,R] \times [-b_2,b_2] \times  [-b_3,b_3] \times \cdots \times  [-b_{n-2},b_{n-2}] \times [-b_{n-1},b_{n-1}]^2.
\end{equation}
For future reference we define 
\[
b:= \left( b_2 \cdots b_{n-1}  \right)^{\frac{1}{n-2}},
\]
and  \eqref{eq:b times R const} becomes
\begin{equation} \label{eq:b small with R}
    b^{n-2} \leq (n-1)! \frac{E_0(\partial\Omega)}{R}.
\end{equation}

The boundary of a bounded convex body $\Omega\subset\R^n$ can be written as the union of $2n$ graphs whose slopes never exceed $\sqrt{n}$ in any direction. For this fact see, for instance, \cite[Lemma 2.1 and 2.2]{Csato2024},  where we use \cite[Lemma 2.2]{Csato2024} with $\epsilon=1$. Moreover, each graph is a function of the $n-1$ variables of one of the canonical hyperplanes $\{x_j=0\}$ for some $j\in\{1,\ldots,n\}.$ This last fact follows immediately from the proof of \cite[Lemma 2.2]{Csato2024}.
More precisely
$$
  \partial\Omega=\bigcup_{i=1}^{2n} M_i\,
$$
and, up to permutations of the canonical axes, $M_i=\{(x',g_i(x'));\, x'\in W_i\subset \R^{n-1}\}$ where we denote with $x=(x',x_n) \in \R^{n-1} \times \R$. Here $g_i$ is some convex or concave function given on $W_i$, which is the projection of $\Omega$ onto $\{x_n=0\}$, such that
\begin{equation}
 \label{eq:nabla g bounded}
 |\nabla g_i|\leq \sqrt{n}.
\end{equation}
We can refine the cover $\{M_i\}_{i=1,\dots,2n}$ into a disjoint partition of $\partial\Omega$ by replacing for $i>1$ each $M_i$ with $M_i \setminus \bigcup_{j=1}^{i-1} M_j$. For the sake of simplicity, we keep the notation $M_i$ for the sets of the disjoint partition, and let $W_i$ denote their respective projections.

As $\Omega$ is contained in $C_{R,b}$ thanks to \eqref{eq:Omega in cylinder}, and $R\geq b_2\geq\ldots\geq b_{n-1}$, the projections $W_i$ of $\Omega$'s boundary  portions satisfy 
\begin{equation}
 \label{eq:new FRb}
  W_i\subset F_{R,b}:=[-R,R]\times[-b_2,b_2] \times \dots \times[-b_{n-2},b_{n-2}]\times [-b_{n-1},b_{n-1}].
\end{equation}
This follows from the fact that these projections are obtained by omitting one of the brackets in the definition of $C_{R,b}$.

From the previous construction  of the cover of $\p \Omega$ given by the disjoint sets $\{M_i\}_{i=1,\dots,2n}$,  and since $\phi$ is increasing, we obtain 
\begin{equation} \label{eq:bound_Ephi_going_down}
\begin{aligned}
 E_{\phi}(\p \Omega)&=\sum_{i=1}^{2n}\int_{W_i} \phi\left(\sqrt{|x'|^2+g_i(x')^2}\right)\sqrt{1+|\nabla g_i(x')|^2} \dd x'\\
 &\geq \sum_{i=1}^{2n}\int_{W_i} \phi\left(|x'|\right)\sqrt{1+|\nabla g_i(x')|^2} \dd x' = \int_{-R}^{ R} \int_{\R^{n-2}}\phi\left(|x'|\right) v(x_1, x'') \dd x'' \dd x_1,
\end{aligned}
\end{equation}
where we split $x' = (x_1, x'') \in  \R \times \R^{n-2}$, and define
 \[
v(x_1, x'') := \sum_{i=1}^{2n} \ \sqrt{1+|\nabla g_i(x_1, x'')|^2} \one_{{W}_i}(x_1, x'').
 \]

We decompose $\phi = \phi^+ - \phi^-$ and argue differently for $\phi^+$ and $\phi^-$. Consider first the positive part $\phi^+$. From \eqref{eq:bound_Ephi_going_down} we use that $\phi^+(|x'|)\geq \phi^+(|x_1|) $  and reach
\begin{equation}\label{eq:phiplus and V}
\int_{-R}^{R} \int_{\R^{n-2}} \phi^+(|x'|) v(x_1,x'') \dd x'' \dd x_1 \geq \int_{-R}^{R}  \phi^+(|x_1|) V(x_1) \dd x_1,
\end{equation}
where 
\[
V(x_1)=\int_{\R^{n-2}} v(x_1, x'') \dd x''.
\]
Our goal is now to apply Lemma \ref{lm:bathtub_principle}. Firstly, notice that 
$$
  \int_{\R} V(x_1) \dd x_1=\int_{-R}^{R} V(x_1) \dd x_1 =E_0(\p \Omega).
$$
Moreover, from \eqref{eq:nabla g bounded}, \eqref{eq:new FRb}, and eventually  using \eqref{eq:b small with R} we bound $V$ by
\begin{equation} \label{eq:V_0}
 \begin{split}
    0 \leq V(x_1) &\leq \int_{\R^{n-2}}\sum_{i=1}^{2n}\sqrt{1+n}\,\,\one_{W_i}(x_1,x'')dx''
    \\
    &\leq 2n\sqrt{1+n}\int_{-b_2}^{b_2}\cdots\int_{-b_{n-1}}^{b_{n-1}} dx''
    \leq C\frac{E_0(\p \Omega)}{R}=:K, 
 \end{split}
\end{equation}
where $C$ is a positive constant depending only on the dimension $n$. 
We are now ready to apply Lemma \ref{lm:bathtub_principle} to $V$ and $\phi^+$ in dimension $k=1$. We assign as a bound for the total mass $D := E_0(\p \Omega)$ and the bound $K$ defined above.  The optimal radius $r$ in Lemma \ref{lm:bathtub_principle} is thus determined by the condition $2r = \frac{D}{K}$, which yields
 \[
 r = \frac{E_0(\p \Omega)}{2K} = \frac{R}{2C} =: \gamma R,
 \]
 for a suitable dimensional constant $\gamma > 0$. Applying Lemma \ref{lm:bathtub_principle} to $V$ and $\phi^+$, we obtain
 \begin{align*}
\int_{-R}^{R}\phi^+\left(|x_1|\right) V(x_1) \dd x_1=
\int_{\R}\phi^+\left(|x_1|\right) V(x_1) \dd x_1 
     &\geq K \int_{-r}^{r} \phi(|x_1|) \dd x_1 = 2K \int_0^{\gamma R} \phi^+(t) \dd t.
 \end{align*}
 Recalling that $2K = \frac{E_0(\p \Omega)}{\gamma R}$, we substitute this into the inequality \eqref{eq:phiplus and V} to reach
 \begin{equation}\label{eq:bound_E_phi+}
        \int_{-R}^{R} \int_{\R^{n-2}} \phi^+(|x'|) v(x_1,x'') \dd x'' \dd x_1  \ge \frac{E_0(\p \Omega)}{\gamma R}\int_0^{\gamma R} \phi^+(t) \dd t.
 \end{equation}
 This proves the estimate on $\phi^+$.

We now focus on bounding the contribution of $\phi^-$ in \eqref{eq:bound_Ephi_going_down}. First observe that for every $x_1$ the map
 $t\in[0,\infty)\mapsto \psi(t):=\phi^-(\sqrt{x_1^2+t^2})$ is a decreasing function
and that $0\leq v \leq K_2$ for the dimensional constant $K_2=2n\sqrt{1+n}$. 
Moreover, for every $x_1\neq 0$ the map $x''\in\R^{n-2}\mapsto \psi(|x''|)$ is locally integrable in $\R^{n-2}$, since $\phi^-$ could only blow up, possibly, at the origin.
Thus, for every $x_1 \in [-R,R]\setminus \{0\}$
we apply Lemma  \ref{lm:bathtub_principle} with $k=n-2$ to $\psi$ and $v$ viewed as functions of $x''$ with the total mass $D=V(x_1)$ and obtain
\begin{equation*}
\int_{\R^{n-2}} \phi^-(|x'|) v(x_1, x'') \dd x'' \leq K_2 \int_{B_{r(x_1)}^{n-2}} \phi^-(|x'|) \dd x''
\end{equation*}
where $r(x_1)$ is given by the relation
\[
r(x_1)=\left( \frac{V(x_1)}{K_2\omega_{n-2}}\right)^{\frac{1}{n-2}} \leq \left( \frac{K}{K_2 \omega_{n-2}}\right)^{\frac{1}{n-2}}.
\]
Recalling the definition of $K$ in \eqref{eq:V_0} and using that $\phi^-\geq 0$ we extend the domain of integration to $B^{n-2}_{r}$ where $r$ is defined as \eqref{eq:def_radius_r}. Integrating in $x_1$ ranging in $[-R,R]$ we obtain
\begin{equation}\label{eq:bound_for_phi-}
\begin{aligned}
    \int_{-R}^{R} \int_{\R^{n-2}} \phi^-(|x'|) v(x') \dd x' &\leq K_2 \int_{-R}^{R} \int_{B^{n-2}_{r}} \phi^-(|x'|) \dd x' 
\end{aligned}
\end{equation}
This concludes the estimate on $\phi^-$.

Combining \eqref{eq:bound_E_phi+}  and \eqref{eq:bound_for_phi-} in \eqref{eq:bound_Ephi_going_down}, we establish the bound  
\begin{equation} \label{eq:bound_Ephi_by_R_inc1}
\frac{E_{\phi}(\p \Omega)}{E_0(\p \Omega)} \geq \frac{1}{\gamma R} \int_0^{\gamma R} \phi^+(t) \dd t - \frac{K_2}{E_0(\p \Omega)} \int_{-R}^R \int_{B^{n-2}_{r}} \phi^-(|x'|) \dd x'\,
\end{equation}
which proves the claimed inequality of the lemma. 

\smallskip

Finally, we analyze the asymptotic behavior for a sequence of convex bodies $(\Omega_i)_{i\geq 1}$ with $E_0(\p \Omega_i)$ fixed and $R_i \to \infty$. We shall distinguish two cases. 

\smallskip

\textit{Case 1: $\sup_{t \ge 0} \phi(t) > 0$.}

If $\sup_{t \ge 0} \phi(t) > 0$, then there exists $t_0 \ge 0$ such that $\phi(t) \ge 0$ for $t \ge t_0$, meaning $\phi^-$ is compactly supported in $[0, t_0]$.
Hence, the function $x'\mapsto \phi^-(|x'|)$ is compactly supported in $B_{t_0}^{n-1}$. Moreover, by our standing integrability assumption for a weight function we have that $t^{n-2}\phi(t) \in L^1([0,1])$. It thus follows that 
$$
   \phi^-(|x'|) \in L^1\big(B_{t_0}^{n-1}\big).
$$

As $R_i \to \infty$, the radius of the bounding cylinder $r_i=(C E_0/R_i)^{\frac{1}{n-2}}$ tends to $0$. Consequently, the set $(-R,R) \times B^{n-2}_{r_i}\cap B_{t_0}^{n-1}\subset (-t_0,t_0)\times B_{r_i}^{n-2}$ shrinks to a set of Lebesgue measure zero in $\R^{n-1}$. 
Hence, the integral representing the negative contribution in \eqref{eq:bound_Ephi_by_R_inc}, i.e., the term containing $\phi^-,$ vanishes as $R_i \to \infty$.

On the other hand, the first term in \eqref{eq:bound_Ephi_by_R_inc} is exactly the integral average of the increasing function $\phi^+$ over the expanding interval $[0, \gamma R_i]$. By the properties of integral averages, this term converges to $\sup_{t \ge 0} \phi^+(t) = \sup_{t \ge 0} \phi(t)$. 

\smallskip

\textit{Case 2: $\sup_{t \ge 0} \phi(t) \leq 0$.}

Let us now deal with the case when the weight function $\phi$ is such that $\sup_{t\geq 0}\phi(t)=-S\leq 0$ for some $S\geq0.$ In this case define $ \phi_S(t):=\phi(t)+S+1$, a new function which satisfies that $\sup\phi_S=1>0$ and also the requirements of the definition of a weight function. Hence we will be able to use the result of Case 1. As we have the relation $E_\phi(\delomega)=E_{\phi_S}(\delomega)-(S+1)E_0(\delomega)$, we can take the limit $i\to\infty$ in
$$
\lim_{i\to\infty}\frac{E_\phi(\delomega_i)}{E_0(\delomega_i)}=\lim_{i\to\infty}\left(\frac{E_{\phi_S}(\delomega_i)}{E_0(\delomega_i)}-(S+1)\right)=1-(S+1)=-S=\sup_{t\geq 0}\phi. 
$$
which proves the claim under the present case. 
\end{customproof}

We are now equipped to prove the main existence result of Theorem \ref{Thm:intro:existence1}. We shall use the notation  $\mathcal{K}^n$ for set of convex bodies:
$$
  \mathcal{K}^n=\{\Omega\subset\R^n:\, \Omega\text{ is a bounded closed convex set}\}.
$$

\begin{customproof}{Theorem \ref{Thm:intro:existence1}}
Let $\{\Omega_i\}_{i\geq1}\subset\mathcal{K}^n$ be either a minimizing or maximizing sequence, depending on which of the two cases $(i)$ or $(ii)$ we are dealing with. 
Before detailing the proof, we observe that we can assume $0 \in {\Omega_i}$ for all $i \ge 1$ without loss of generality. Indeed, if $0 \notin {\Omega_i}$, translating the set towards the origin decreases the Euclidean norm of all its points, thanks to the convexity of $\Omega_i$. This operation improves the energy $E_\phi$ in both optimization problems (minimizing an increasing weight or maximizing a decreasing weight) while leaving the perimeter $E_0$ unchanged.

\smallskip
\textit{Proof of $(i)$.} Let $\phi$ be an increasing function, and let $\{\Omega_i\}_{i\geq 1}\subset\mathcal{K}^n$ be a minimizing sequence for $E_\phi(\p \Omega)$ among convex bodies with $E_0( \p \Omega)=\lambda$, for some $\lambda>0$. 
That is,
$$
\inf\{ E_\phi(\p \Omega) \, : \, \Omega \in \mathcal{K}^n, E_0( \p \Omega)=\lambda\}=\lim_{i\to\infty}E_\phi(\partial\Omega_i), 
$$
and $E_0( \p \Omega_i)=\lambda$ for all $i\geq 1$.

\textit{Claim 1.}  We claim that the sequence of sets $\{\Omega_i\}_{i\geq 1}$ is uniformly bounded; that is, there exists $R>0$ such that $\Omega_i\subset B_R$ for all $i\geq 1$. For this purpose we can assume
without loss of generality that $\phi$ is not a constant function, otherwise there is nothing to be proven.
We argue by contradiction and assume the sequence is unbounded. 
Since we established that $0 \in {\Omega_i}$ for all $i\geq1$, the sequence cannot escape to infinity while maintaining a uniformly bounded diameter. Thus, its unboundedness necessarily implies (up to extracting a subsequence) that the diameters must diverge
\begin{equation}\label{eq:gy:diam infty}
\lim_{i\to\infty}\text{diam}(\Omega_i)=\infty.
\end{equation}

If $n=2$, this immediately yields a contradiction, as the diameter of a convex set is universally bounded by its perimeter, which is equal to  $1$ for all $i$. On the other hand, for $n\geq 3$, we will deduce that \eqref{eq:gy:diam infty} cannot hold in virtue of Lemma \ref{lm:bound_on_max_radius_decreasing}, which is applicable since $0 \in {\Omega_i}$. By defining $R_i := \max_{x \in \partial\Omega_i} |x|$, the unbounded diameter condition implies $R_i \to \infty$. Thus Lemma \ref{lm:bound_on_max_radius_decreasing} yields that $E_\phi(\p \Omega_i) \to \sup_{t\geq 0} \phi(t)$. This contradicts the assumption that the sequence is minimizing, as any  competitor whose boundary contains the origin provides a smaller energy if $\phi$ is nonconstant. This creates a contradiction and settles Claim 1.

From the Blaschke selection theorem, see for instance \cite[Theorem 1.8.7]{Schneider}, we infer that, up to a subsequence, $\Omega_i$ converge in the Hausdorff metric to a limiting compact convex set $\Omega_\infty$. The existence of the minimizer is then readily established by combining the continuity of the perimeter and of the energy $E_\phi$, which are given by Theorem \ref{lm:continuity_energy}.

\smallskip
\textit{Proof of $(ii)$.} The existence of a maximizer when $\phi$ is decreasing follows  from Part $(i)$ applied to $-\phi$.

\smallskip
Finally, in the particular case of the specific weights $E_{p}$ and $E_{-\alpha}$, the existence of extremals gives the inequalities \eqref{eq:gy:universal constant p pos} and \eqref{eq:gy:universal constant p neg} due to homogeneity or scaling.

The fact that equality is achieved by some bounded open convex set under the  conditions $p> 4-n$, and $n \ge 5$ with $\alpha < n-4$ is Theorem \ref{thm:nonempty_interior}, which we prove in the next section. 
\end{customproof}

\subsection{Non-degeneracy of extremals via perturbations of cylinders}
\label{Sec:non_degeneracy}

In this section we investigate the conditions under which we can rule out the degeneracy of the extremals, i.e. guarantee that they possess a nonempty interior. Standard rearrangement arguments show that if an optimal convex set, meaning minimizer for $E_p$, respectively maximizer for $E_{-\alpha}$, is degenerate, it must necessarily  be an $(n-1)$-dimensional disk. Therefore, our strategy is to perturb a degenerate disk by thickening it into a cylinder while strictly preserving the perimeter constraint, and to analyze the corresponding variations of the energy.

\begin{Theorem} \label{thm:nonempty_interior}
Let $n\geq 2$ be an integer. Then the following properties hold:
   \begin{itemize}
       \item[(i)] If $p> \max\{4-n,0\}$, then the minimizer of $E_p$ of the problem \eqref{min_problem_p_pos} has a nonempty interior.
       \item[(ii)] If $n \ge 5$ and $0 < \alpha < n-4$, then the maximizer of $E_{-\alpha}$ of the problem \eqref{max_problem_p_neg} has a nonempty interior.
   \end{itemize} 
\end{Theorem}

\begin{Remark}
    Condition (i) implies that for $E_p$ every minimizer is nondegenerate whenever $n \ge 4$. In lower dimensions, nondegeneracy is guaranteed for $p>1$ when $n=3$, and for $p> 2$ when $n=2$. But we can include the case $p=2$ and $n=2$ too, since it is known that the minimizer is the equilateral triangle centered at the origin due to  Sachs \cite{SachsIandII} and Hall \cite{Hall1985}. Concerning $E_{-\alpha}$, observe that in low dimensions $n \le 4$, the perturbation of a flat cylinder cannot exclude the degeneracy of the maximizer for any given $0<\alpha<n$. This limitation strongly aligns with our conjecture that highly singular weights force the optimal convex body to be a lower-dimensional degenerate set.
\end{Remark}

To prove Theorem \ref{thm:nonempty_interior}, we construct a family of continuous $1$-parameter variations consisting of cylinders $\Omega_r$ of radius $r \in (0,1]$, designed to keep the surface area constant. The limit case $r=1$ corresponds exactly to the degenerate double disk $\Omega_1=\Omega$. The boundary of these cylinders is decomposed as $\partial\Omega_r = M_r \cup D_r$, where $M_r$ represents the lateral surface contained in $\{(x',x_n)\in\R^n:\,x'\in\R^{n-1}\text{ and }|x'|=r\}$ and $D_r$ represents the two identical flat disks (bases), contained in two planes parallel to the plane $x_n=0$.

To streamline the presentation, throughout the proof we unify the notation by letting $p \in (1-n, \infty)$, thus treating the regular ($p>0$) and singular ($p = -\alpha < 0$) cases simultaneously.

For convenience we set the surface area of $\Omega_r$ to be $2 \omega_{n-1}=2\mathcal{H}^{n-1}(B_1^{n-1})$, where $B_1^{n-1}$ is the unit ball in $\R^{n-1}$ and define $\Omega_r$ as
\[
\Omega_r := B_{r}^{n-1} \times [-h(r),h(r)].
\]

The constraint $2 \omega_{n-1}=\mathcal{H}^{n-1}(\p \Omega_r)= \mathcal{H}^{n-1}(M_r)+ \mathcal{H}^{n-1}(D_r)=2 h(r) \beta_{n-2}r^{n-2}+2\omega_{n-1}r^{n-1}$ yields the following explicit expression of the function $h(r)$
\[
h(r)=\frac{\omega_{n-1}}{\beta_{n-2}} \left(r^{2-n}-r \right)=\frac{r^{2-n}-r}{n-1}.
\]
Note that
$$
  h(1)=0,\quad h'(1)=-1\quad \text{ and }\quad h''(1)=n-2. 
$$

The total energy functional splits into the contributions from the lateral surface and the bases and it is equal to 
\begin{equation}
    E_p(\p \Omega_r) = \int_{M_r} |x|^{p} \dH(x) +  \int_{D_r} |x|^{p} \dH(x) =: 2\beta_{n-2}E_p(M_r) + 2\beta_{n-2} E_p(D_r),
\end{equation}
where the energy of the  lateral surface $M_r$ (modulo the factor $2\beta_{n-2}$)  is given by the integral
\begin{equation}
    E_p(M_r): = r^{n-2} \int_0^{h(r)} (r^2 + x_n^2)^{\frac{p}{2}} \dd x_n .
\end{equation}
Similarly, for the bases $D_r$, the energy $E_p(D_r)$ is defined by
\begin{equation}
    E_p(D_r) :=  \int_0^r t^{n-2} (t^2 + h(r)^2)^{\frac{p}{2}} \dd t .
\end{equation}
Note that $E_p(\p\Omega_r)$ is a continuous function of $r$ on $(0,1]$ and smooth on $(0,1).$ As a convenient abbreviation in this section we will define
$$
  E_p':= \frac{1}{2\beta_{n-2}}\lim_{r\nearrow 1} \frac{d}{dr} E_p(\partial\Omega_r)  .
$$
Recall that in this section we deviate from our convention that $p$ is always positive, as it does not matter for the mere calculation of derivatives whether we eventually maximize or minimize the energy ($E_{-\alpha}$ or $E_p$). 

\begin{Lemma}\label{lemma:first variation of Ep}
Let $n\geq 2$ be an integer and $ p\in(1-n,\infty)\setminus\{0\}$. Then 
\begin{equation} \label{eq:thesis_first}
    E_p' =\begin{cases}
          \infty   & \text{ if } p<2-n, \\
      >0 & \text{ if } p = 2-n, \\
    0 & \text{ if } p>2-n. \\
      \end{cases}
\end{equation}
\end{Lemma}

Note that if $p<0$ and $p\leq 2-n,$ then Lemma \ref{lemma:first variation of Ep} gives support to Conjecture \ref{conjecture:intro:alpha}.

\smallskip

\begin{proof}
To find the first variation, we differentiate the energies $E_p(M_r)$ and $E_p(D_r)$ with respect to $r$. The derivative of the lateral surface contribution is
\begin{align*}
    \frac{d}{dr}E_p(M_r) &= (n-2)r^{n-3}\int_0^{h(r)}(r^2+x_n^2)^{\frac{p}{2}}\dd x_n  \\
    &\quad + r^{n-2}(h(r)^2+r^2)^{\frac{p}{2}}h'(r)  \\
    &\quad + p r^{n-2}\int_0^{h(r)}(r^2+x_n^2)^{\frac{p}{2}-1}r \,  \dd x_n .
\end{align*}
Evaluating this at $r=1$ (where the cylinder height vanishes, $h(1)=0$, and $h'(1)=-1$), the first and third terms vanish, yielding
\begin{equation}\label{Mr prime}
    \lim_{r\nearrow 1}  \frac{d}{dr}E_p(M_r)  = h'(1) = -1.
\end{equation}

Next, we compute the derivative of the bases' contribution
\begin{equation}\label{EpDr prime}
    \frac{d}{dr}E_p(D_r) =  r^{n-2}(r^2+h(r)^2)^{\frac{p}{2}} + p h(r)h'(r) \int_0^r t^{n-2}(t^2+h(r)^2)^{\frac{p}{2}-1}   \dd t .
\end{equation}
Observe that the asymptotic behavior of the second term as $r \nearrow 1$ is not immediately obvious for $p\leq 3-n$.

However, for $p>3-n$ the second term vanishes as $r \nearrow 1$ and thus
$$
  \lim_{r\nearrow 1} \frac{d}{dr}E_p(D_r) = 1\quad\text{ if } p>3-n.
$$

To study the case $p\leq 3-n$ we make the change of variables $t=h(r) s$ in the integral on the right side of \eqref{EpDr prime} and obtain
\begin{align*}
    I(r):= h(r)h'(r) \int_0^r t^{n-2}(t^2+h(r)^2)^{\frac{p}{2}-1}   \dd t = h(r)^{n+p-2}h'(r) \int_0^{\frac{r}{h(r)}} s^{n-2}(s^2+1)^{\frac{p}{2}-1}   \dd s. 
\end{align*}
Here note that as $r \nearrow 1$ the upper limit of integration $r/h(r) \to \infty$. However, for $p <3-n$, the integral remains bounded at infinity. Thus, we have
\begin{equation}\label{eq:limit Ir}
    \lim_{r \nearrow 1} I(r) =- \lim_{r \nearrow 1}  h(r)^{n+p-2} C_{n,p}=\begin{cases} -\infty & \text{ if } p<2-n, \\
    -C_{n,p} & \text{ if } p = 2-n, \\
     0 & \text{ if } 2-n < p < 3-n, \\
    \end{cases}
\end{equation}
where, using $p<3-n$, 
\[
C_{n,p}=\int_0^{\infty} s^{n-2}(s^2+1)^{\frac{p}{2}-1}   \dd s >0.
\]
Let us now analyze the left side of \eqref{eq:limit Ir} when $p=3-n$. In this case we need to deal with 
\[
h(r)\int_0^{\frac{r}{h(r)}} s^{n-2}(s^2+1)^{\frac{3-n}{2}-1}   \dd s.
\]
We split the domain of integration into $(0,1)$ and $(1,r/h(r)).$ 
The integral close to origin is finite and thus $h(r) \int_0^{1} s^{n-2}(s^2+1)^{\frac{3-n}{2}-1}   \dd s \to 0$ as $r \nearrow 1$. For the integral on $(1,r/h(r))$ we have, as $n\geq 2$, 
\[
\int_1^{\frac{r}{h(r)}} s^{n-2}(s^2+1)^{\frac{3-n}{2}-1}   \dd s \leq  \int_1^{\frac{r}{h(r)}} s^{n-2}s^{3-n-2}   \dd s = \int_1^{\frac{r}{h(r)}}s^{-1}\dd s =  \log \left( \frac{r}{h(r)}\right).
\]
From  $h(r)\log \left( \frac{r}{h(r)}\right) \to 0$ as $r \nearrow 1$, we conclude that 
$$
  \lim_{r\nearrow 1}I(r)=0\quad\text{ if } p=3-n.
$$
Collecting the above results, evaluating the first variation of $D_r$ at $r=1$, we obtain
\begin{equation*}
      \lim_{r\nearrow 1} \frac{d}{dr} E_p( D_r)  =  1 + p \lim_{r \nearrow 1} I(r)= \begin{cases}
          \infty   & \text{ if } p<2-n, \\
     1 -p  C_{n,p} >0 & \text{ if } p = 2-n, \\
     1& \text{ if } p>2-n. \\
      \end{cases}
\end{equation*}
Summing the contribution \eqref{Mr prime} of $M_r$ at $r=1$ leads to \eqref{eq:thesis_first}.

\end{proof}

If $p>2-n$, we need to  compute the second variation of $E_p(\p \Omega_r)$, since the first variation gives no information. 

\begin{Lemma}\label{lm:second_variation_Ep}Let $n\geq 2$ and $p>2-n$ such that $p\neq 0$. Then$$  E_p'':=\frac{1}{2\beta_{n-2}}\lim_{r\nearrow 1}  \frac{d^2}{dr^2} E_p(\partial\Omega_r) 
  =\begin{cases}
    -p\,\displaystyle{\frac{p+n-4}{p+n-3}} & \text{ if } p>3-n, \\[3mm]
    -\infty & \text{ if } 2-n < p \le 3-n \text{ and } n \ge 3, \\[2mm]
    +\infty & \text{ if } 0 < p \le 1 \text{ and } n = 2.
   \end{cases}$$
   \end{Lemma}

\begin{proof} We will use the formulas obtained for the first variation in the proof of Lemma \ref{lemma:first variation of Ep}.

The second variation of the energy on the lateral surface gives
\begin{align*}
    \frac{d^2}{dr^2}E_p(M_r) &= (n-2)(n-3)r^{n-4}\int_0^{h(r)}(r^2+x_n^2)^{\frac{p}{2}}\dd x_n  \\
    &+(n-2)r^{n-3}(r^2+h(r)^2)^{\frac{p}{2}}  h'(r)  \\
    &+(n-2)r^{n-3} \int_0^{h(r)}\frac{d}{dr} \left((r^2+x_n^2)^{\frac{p}{2}} \right)\dd x_n  \\
    &+ (n-2)r^{n-3}(h(r)^2+r^2)^{\frac{p}{2}}h'(r)  \\
    &+ r^{n-2}\left[p(h(r)^2+r^2)^{\frac{p}{2}-1} \left( r+h(r)h'(r)\right)h'(r) +(h(r)^2+r^2)^{\frac{p}{2}}h''(r) \right] \\
    &+  p (n-2)\,  r^{n-3} \int_{0}^{h(r)}  r \left( r^2 + t^2 \right)^{\frac{p}{2}-1} \dd t \\
    &+  p \,  r^{n-2} r\left( r^2 + h^2(r)\right)^{\frac{p}{2}-1}h'(r) \\
    &+ p \, r^{n-2} \int_0^{h(r)} \frac{d}{dr} \left( (r^2 +t^2)^{\frac{p}{2}-1} r \right) \dd t.
\end{align*}

When plugging $r=1$ all the terms having the integral vanish since $h(1)=0$ and we are left with
\begin{align}\label{EpMr double prime}
   \lim_{r\nearrow 1}  \frac{d^2}{dr^2}E_p(M_r) = 2(n-2)h'(1) + 2p h'(1) +h''(1)   =  2-n-2p,
\end{align}
since $h'(1)=-1$ and $h''(1)=n-2$.

Computing the second variation contribution of the bases $D_r$ yields
\begin{align*}
    \frac{d^2}{dr^2}E_p(D_r)&= (n-2) r^{n-3} \left(r^2 +h(r)^2 \right)^{\frac{p}{2}}   \\
    &+ p r^{n-2} \left(r^2 +h(r)^2 \right)^{\frac{p}{2}-1} \left(r+h(r)h'(r) \right)   \\
    &+ p   \left( h(r)h''(r)+h'(r)^2\right) \int_0^r  \left( t^2 +h(r)^2 \right)^{\frac{p}{2}-1} t^{n-2} \dd t  \\
    &+ph(r)h'(r) \left[ \left(r^2+h(r)^2 \right)^{\frac{p}{2}-1} r^{n-2} + (p-2) h(r)h'(r)\int_{0}^{r} \left( t^2 +h(r)^2 \right)^{\frac{p}{2}-2} t^{n-2} \dd t\right].
\end{align*}
Using that $h(1)=0,$ $h'(1)=-1,$ and $h''(1)=n-2$ we can write the second derivative of $E_p(D_r)$ as
\begin{align}\label{EpDr double prime with limit}
    \lim_{r\nearrow 1}\frac{d^2}{dr^2}E_p(D_r) &= \lim_{r\nearrow 1} \left\{ n-2 + p+p [(n-2)h(r)J_1(r)+J_1(r) + (p-2)J_2(r)]\right\},  
     \end{align}
where
\begin{align*}
    J_1(r)&:= \int_0^r  \left( t^2 +h(r)^2 \right)^{\frac{p}{2}-1} t^{n-2} \dd t, \\
    J_2(r)&:= h(r)^2 \int_{0}^{r} \left( t^2 +h(r)^2 \right)^{\frac{p}{2}-2} t^{n-2} \dd t.
\end{align*}

To analyze the above expression as $r \nearrow 1$, we consider different ranges of $p$.

\textit{\underline{Case $p>5-n$.}} 
In this case, all terms in \eqref{EpDr double prime with limit} are finite and thus taking the limit gives

     \begin{align*}
  \lim_{r\nearrow 1} \frac{d^2}{dr^2}E_p(D_r) &=   n-2 + p+p  \int_0^{1} t^{p+n-4} \dd t=  n-2+p + \frac{p}{p+n-3}.
 \end{align*}

Thus summing with the contribution \eqref{EpMr double prime} of $M_r$ leads to 

\begin{align}\label{eq:E_p_doublep>5-n}
    E_p'' &=   2-n-2p +   n-2+p + \frac{p}{p+n-3} =-p \; \frac{p+n-4}{p+n-3}.
\end{align}
\textit{\underline{Case $3-n<p\leq 5-n$.}}  
In this case we have $\lim_{r\nearrow 1}J_1=\frac{1}{p+n-3}$. In order to deal with $J_2$ we use the change of variables $t=h(r)s$ and obtain
\begin{align}\label{eq:J_2_contr}
    J_2(r)=h(r)^{n+p-3} \int_0^{\frac{r}{h(r)}} s^{n-2} (1+s^2)^{\frac{p}{2}-2} \dd s.
\end{align}
Whenever $p <5-n$ the integral is finite and thus $\lim_{r \nearrow 1} J_2=0$. In the borderline case $p=5-n$ for $s>1$ we have
\[
\int_{1}^{\frac{r}{h(r)}} s^{n-2} (1+s^2)^{\frac{p}{2}-2} \dd s \leq \int_{1}^{\frac{r}{h(r)}} s^{-1} \dd s  = \log \left( \frac{r}{h(r)}\right).
\]
 Computing $\lim_{r \nearrow 1} h(r)^2  \log \left( \frac{r}{h(r)}\right)$ reveals that $\lim_{r \nearrow 1} J_2$ vanishes.
 
 This implies that also in this range we obtain \eqref{eq:E_p_doublep>5-n}.

\smallskip  
\textit{\underline{Case $p=3-n$.}}
In this case, we analyze the exact explosive behavior of $J_1(r)$. Indeed, from \eqref{eq:J_2_contr} we deduce that the term $J_2(r)$ gives a finite contribution. For $J_1(r)$, using once again $t=h(r)s$, we obtain
\begin{align} \label{eq:J_1}
    J_1(r) = h(r)^{p+n-3} \int_0^{\frac{r}{h(r)}} (s^2+1)^{\frac{p}{2}-1} s^{n-2} \dd s = \int_0^{\frac{r}{h(r)}} (s^2+1)^{\frac{1-n}{2}} s^{n-2} \dd s.
\end{align}

As $r\nearrow 1$, the integrand behaves asymptotically as $s^{-1}$. Thus, $J_1(r)$ diverges  as $\log(r/h(r))$.
On the other hand $h(r)J_1(r)$ vanishes as $r \nearrow 1$.
Therefore, since the contribution of the lateral surface is finite, \eqref{EpDr double prime with limit} shows that the behavior of $E_p''$ is uniquely determined by the unbounded term $p J_1(r)$.
For $n>3$, since $p=3-n<0$, we get $E_p''=-\infty$.
While if $n = 2$, then $p = 1$ and we obtain that $E_p'' = \infty$.

\smallskip
\textit{\underline{Case $2-n<p<3-n$.}} 
In the remaining case where $2-n < p < 3-n $  we have that both $J_1$ and $J_2$ are unbounded as $r \nearrow 1$. For this, combining \eqref{EpMr double prime} and \eqref{EpDr double prime with limit}, we express $E_p''$ as
\begin{equation} \label{asymptotic_d2E}
 E_p'' = -p+p[(n-2)h(r)J_1(r)+J_1(r)+(p-2)J_2(r)] +  o(1)\quad \text{as }r \nearrow 1.
\end{equation}
To deal with the case $p \in (2-n,3-n)$ we express conveniently $J_1(r) + (p-2)J_2(r)$. For this we note that (using simply $n+p-3=(n-1)+(p-2)$ in the second identity)
\begin{align*}
	\frac{\dd}{\dd t} \left[ t^{n-1} (t^2 + h^2(r))^{\frac{p}{2}-1}\right] &= (n-1)t^{n-2} (t^2 + h^2(r))^{\frac{p}{2}-1} + (p-2)t^{n}(t^2 + h^2(r))^{\frac{p}{2}-2} \\
	&=(n+p-3) t^{n-2}(t^2 + h^2(r))^{\frac{p}{2}-1} - (p-2) h^2(r) t^{n-2}(t^2 + h^2(r))^{\frac{p}{2}-2}.
\end{align*}
Note that both terms on the right hand side are integrable in $t$ on the interval $(0,r)$ for every $r<1$ for $n\geq 2.$ Thus, 
integrating both sides of the above expression with respect to $t$ from $0$ to $r$ yields the relation
\[
r^{n-1}(r^2 + h^2(r))^{\frac{p}{2}-1} =(n+p-3) J_1(r) -(p-2)J_2(r)
=(n+p-2)J_1(r)-J_1(r)-(p-2)J_2(r),
\]
which implies
\[
J_1(r) +(p-2)J_2(r) = (n+p-2)J_1(r) -r^{n-1}(r^2 + h^2(r))^{\frac{p}{2}-1} 
\]

Thus, as $r \nearrow 1$, we can further simplify \eqref{asymptotic_d2E} as
\begin{align} \label{eq:final1}
 E_p'' = -p+p\left[(n-2)h(r)J_1(r)+(n+p-2)J_1(r) -r^{n-1}(r^2 + h^2(r))^{\frac{p}{2}-1} \right] + o(1)
\end{align}
as $r \nearrow 1.$ The last term in the brackets $[\ldots]$ tends to $1$ as $r \nearrow 1$.  For $n\neq 2$, using that $p<3-n$ we have that
$$
  \lim_{r\nearrow 1}J_1(r)=\infty.
$$
Similarly, for the first term in \eqref{eq:final1}, since $p>2-n$, we get as in \eqref{eq:J_1}
\[
\lim_{r\nearrow 1} h(r) J_1(r) = \lim_{r\nearrow 1}  h^{n+p-2}(r) \int_0^{\frac{r}{h(r)}} s^{n-2} (1+s^2)^{\frac{p}{2}-1} \dd s =0.
\]
Hence, since $p < 0$ and $(n+p-2) > 0$, we conclude
\[
E_p''= -\infty\qquad \text{ if }p\in (2-n,3-n)\quad\text{ and }p\neq 2.
\]

For $n=2$ and $0<p<1$ the first term in \eqref{eq:final1} vanishes.  In this case $ E_p''$  has the same behavior of $
\lim_{r \nearrow 1}p^2J_1(r)$.
Since $0<p<1$ we obtain $E_p''=\infty$ due to
\[
\lim_{r \nearrow 1}J_1(r) = \lim_{r\nearrow 1}  h^{p-1}(r) \int_0^{\frac{r}{h(r)}} s^{n-2} (1+s^2)^{\frac{p}{2}-1} \dd s =\infty,
\]
which concludes the proof in this case. 
\end{proof}

We are now ready to prove the main theorem of this section. 

\begin{customproof}{Theorem \ref{thm:nonempty_interior}}
We proceed by contradiction. Assume that the extremal body $\Omega^*$ (the minimizer for $E_p$ if $p>0,$  or the maximizer for $E_{-\alpha}$ given by Theorem \ref{Thm:intro:existence1}) has empty interior. As previously noted, such a convex body must be a flat disk $\Omega_1$ contained in $\R^{n-1}$, by standard symmetrization techniques. 
We have seen that the first variation of $E_p$ computed in Lemma \ref{lemma:first variation of Ep} does not give any contradiction in the case $p\leq 2-n \leq 0$ and supports the Conjecture \ref{conjecture:intro:alpha}. So let us assume $p>2-n.$ We will distinguish three cases, corresponding to the three cases of the second variation formula  of Lemma \ref{lm:second_variation_Ep}. However, the last two cases in Lemma \ref{lm:second_variation_Ep} are not strictly necessary for the proof, as they do not yield a contradiction. Nevertheless, we include them for completeness. These other cases show that for almost all the other ranges of $p$ and $\alpha$ complementing the conditions of Theorem 
 \ref{thm:nonempty_interior} the flat disc $\Omega_1=\Omega^*$ is indeed a local minimizer for $E_p$ if $p>0,$ respectively a local maximizer for $E_{-\alpha}$  among the cylinders $\{\Omega_r\}_{r\in (0,1]}$.  In the exceptional case $p=0$ or $p=4-n$ the second variation vanishes too and no information can be obtained. 
 
\textit{\underline{Case $n=2$ and $0<p\leq 1$}} 

In this case we wish to minimizer $E_p$, hence $E_p'=0$ and $E_p''=+\infty$ shows that $\Omega^*$ is a local minimizer among the cylinders. Thus there is no contradiction. 

\textit{\underline{Case $n\geq 3$ and $2-n<p\leq 3-n$.}} 

In this case $p\leq 0$ and we wish to maximizer $E_p$. Hence $E_p'=0$ and $E_p''=-\infty$ shows that $\Omega^*$ is indeed a local maximizer among the cylinders. Again there is no contradiction. 

\textit{\underline{Case $p>3-n$.}} 

In the range $p>3-n$ the sign of $E_p''$ is given by the following table 
$$
     E_p'' =-p \; \frac{p+n-4}{p+n-3}=\left\{ \begin{array}{clr}
              <0 & \text{ if } p <0, \, 3-n<p<4-n, & \\
              >0 & \text{ if } p <0,\, p>4-n, & \text{(a)} \\
              =0 & \text{ if } p=0 \text{ or } p=4-n,& \\
              >0 & \text{ if } p >0, \, 3-n<p<4-n,& \\
              <0 & \text{ if } p >0, \, p>4-n. & \text{(b)}\\
    \end{array}\right.
$$

Among the $5$ situations listed above, only (a) and (b) give a contradiction. For instance in  case (a) we have assumed that $\Omega^*$ is a maximizer (as $p<0$), but this contradicts the fact that $E_p''>0.$  Whereas in case (b), the second variation is strictly negative which contradicts the fact that $\Omega^*$ is a minimizer (as for $p>0$ we are minimizing $E_p$).

\end{customproof}

\subsection{Extremals of $E_{-\alpha}$ have origin on the boundary}\label{Sec:qualitative}

It is interesting to note that among the transformations preserving both convexity and perimeter, translation is the most elementary. Nevertheless, these shifts provide valuable information regarding the optimal sets. Indeed, when the weight of $E_{-\alpha}$ is highly singular, that is, $\alpha \in (n-2,n-1)$, then the origin has to lie on the boundary of the maximizer. Note that in the case $n=2$, this gives the full range of admissible $\alpha$. The same property holds for decreasing weights that are strictly subharmonic. This is the content of Theorem \ref{thm:shifting_generalized} which we now prove.

\begin{customproof}{Theorem \ref{thm:shifting_generalized}}
Note first that we can assume that $\Omega$ is nondegenerate, that is, it has nonempty interior, and that $0\notin\Omega^c$.   If $\Omega$ is degenerate, and $0\notin \Omega=\delomega$, then one can shift $\Omega$ towards the hyperplane parallel to $\Omega$ and containing the origin  $0$, and the energy $E_{\phi}$ would increase along this shift since $\phi$ is decreasing. Thus for degenerate $\Omega$ the result is obvious. Using that the weight is decreasing, a similar argument shows that if $0\in \Omega^c$, then $\Omega$ cannot be a maximizer.

We now argue by contradiction to rule out the case $0 \in \interior(\Omega)$.  Recall that $f(x):=\phi(|x|)$ and, since the distance between the origin and $\p \Omega$ is positive, all the derivatives of $f$ considered below are well-defined and continuous on $\p\Omega$.

Choose a direction $e \in \p B_1$. For $t>0$ consider a translation of $\Omega$ by an amount $t$ in the direction $e$, which we denote by $\Omega_t(e) = \Omega + te$. Observe that 
\[
E_{\phi}(\p \Omega_t(e)) = \int_{\p \Omega} f(x+te) \dH(x).
\]
 We want to consider variations of a maximizer under this translation and thus it will be convenient to differentiate the integrand of $E_{\phi}(\p \Omega_t(e))$ with respect to $t$.
The derivatives are
\begin{align*}
\frac{d}{dt} \big( f(x+te) \big) = \nabla f(x+te) \cdot e, \quad
\frac{d^2}{dt^2} \big( f(x+te) \big) = e^T \, D^2 f(x+te) \, e. 
\end{align*}
 By computing the first variation and second variation at $t=0$ of  $E_{\phi}$ we obtain
\begin{align}  \label{eq:first-variation_gen}
\frac{d}{d t} E_{\phi}(\p \Omega_t(e)) \Big|_{t=0} &=  \int_{\p \Omega} \nabla f(x) \cdot e    \dH(x) ,
\\ \label{eq:variations_for_t_0_gen}
\frac{d^2}{d t^2} E_{\phi}(\p \Omega_t(e)) \Big|_{t=0} &= \int_{\p \Omega}  e^T \, D^2 f(x) \, e \dH(x).
\end{align}

If $\Omega$ is a maximizer we have that 
\[
\left(\frac{d^2}{d t^2} E_{\phi}(\p \Omega_t(e)) \right)\!\Big|_{t=0} \leq 0 \quad \text{for any }e \in \p B_1.
\]
In particular, summing over all canonical directions $e_i$ for $i=1,\dots,n$ and using the expression given in \eqref{eq:variations_for_t_0_gen} we get

\[
\sum_{i=1}^{n} \left(\frac{d^2}{d t^2} E_{\phi}(\p \Omega_t(e_i)) \right)\!\Big|_{t=0} = \int_{\p \Omega} \Delta f(x)\dH(x) \leq 0.
\] 
This creates a contradiction since $\Delta f$ is a positive continuous function in $\R^n \setminus \{0 \}$ and $0 \notin \p \Omega$. Thus $0$ must lie on $\p \Omega$.

The case of  $\Omega$  being a minimizer and $\phi\in C^2(0,\infty)$ an increasing weight  with $\Delta f<0$ in $\R^n \setminus \{ 0\}$ follows by applying the above argument to $-\phi$.

The weight $|x|^{-\alpha}$ satisfies the hypotheses of the theorem since $\Delta \left( |x|^{-\alpha}\right) = \alpha(\alpha+2-n)|x|^{-\alpha-2} > 0$ for $n-2 < \alpha < n-1$.
\end{customproof}

For $\alpha \in (0,n-2)$ (or $p>0$ when considering $E_{p}$), the same reasoning of the proof of Theorem \ref{thm:shifting_generalized} implies that the optimizers satisfy a rigidity condition resembling a barycenter condition. We will call this property an $\alpha$-weighted (or $p$-weighted) barycentric condition. We state it compactly below, considering the energy involving the singular kernel for negative $p$ too, that is, $E_{-\alpha}$.

\begin{Corollary} \label{cor:weighted_barycenter}
    Let $n\geq 2$ and  $p>2-n$. If $\Omega \subset \R^n$ is critical for $E_{p}$ among all convex bodies  of equal perimeter $E_0( \p \Omega)=\lambda$, for some $\lambda>0$,  then $\p \Omega$ satisfies a $p$-weighted barycentric condition, namely
\begin{equation} \label{weighted_bar_condition}
    \int_{\p \Omega} |x|^{p-2}x \dH(x) = 0
\end{equation}
\end{Corollary}
\begin{proof}
The proof is based on the computations in the proof of Theorem \ref{thm:shifting_generalized}. Indeed, whenever $p > 2-n$, the first variation, see \eqref{eq:first-variation_gen}, remains bounded. Hence, if $\Omega$ is critical for $E_{p}$, then for all $e \in \p B_1$  the following condition holds
\[
0=\frac{d}{d t} E_{p}(\p \Omega_t(e)) \Big|_{t=0} = p  \left( \int_{\p \Omega} |x|^{p-2}  x \,  \dH(x)  \right) \cdot e
\]
which implies \eqref{weighted_bar_condition}.
\end{proof}

\section{Shape of extremals among symmetric curves in the plane}

\subsection{Maximizing $E_{-\alpha}$ in  $\R^2$}
\label{section:singular}

In this section we present the first proof we have found of Theorem \ref{thm:intro:alpha and n is 1}, respectively of the inequality \eqref{intro:functional ineq}. It is more elementary and shorter than the second proof, that is, that of Theorem \ref{Thm:general_phi}.

Let $\Gamma_1=\Gamma \cap \{(x,y)\in\R^2;\, x\geq 0,\, y\geq 0\}$ be the part of  the convex curve $\Gamma$ within the first quadrant. Then, due to the symmetry assumption on $\Gamma$, Theorem \ref{thm:intro:alpha and n is 1} is equivalent to the following inequality
\begin{equation}\label{eq:quadrant_lemma_alpha_neg}
  \int_{\Gamma_1} |x|^{-\alpha} \dHo(x) \leq  \frac{1}{1-\alpha}|\Gamma_1|^{1-\alpha}.
  \end{equation}
  If there is equality, then $\Gamma_1$ is an interval on the $y$-axis, that is, $\Gamma_1=\{0\} \times [0,|\Gamma_1|].$

In the first lemma we deal with the special case of inequality \eqref{eq:quadrant_lemma_alpha_neg} where the curve in the first quadrant is a straight line. This lemma will only be used to quickly address the equality case in Theorem \ref{thm:intro:alpha and n is 1}.

\begin{Lemma}\label{gy:lemma:quadrant straight line}
Let $r\in [0,1]$ and $\Gamma=\{(rt,1-t);\, t\in [0,1]\}.$ Then \eqref{eq:quadrant_lemma_alpha_neg} holds and there is equality in \eqref{eq:quadrant_lemma_alpha_neg} if and only if $r=0.$
\end{Lemma}

\begin{proof}
Under the current assumptions \eqref{eq:quadrant_lemma_alpha_neg} is equivalent to
\[
\int_0^1 \left( r^2 t^2 +(1-t)^2 \right)^{-\alpha/2} \dd t \leq \frac{(1+r^2)^{-\alpha/2}}{1-\alpha}.
\]
Since the map $t \mapsto t^{-\alpha/2}$ is strictly convex we have
\begin{equation}\begin{split}\label{gy:eq:straight line}
    \left( r^2 t^2 +(1-t)^2 \right)^{-\alpha/2} &= (1+r^2)^{-\alpha/2} \left( \frac{r^2}{1+r^2} t^2 + \frac{1}{1+r^2}(1-t)^2 \right)^{-\alpha/2} \\
    &\leq (1+r^2)^{-\alpha/2} \left( \frac{r^2}{1+r^2} t^{-\alpha} + \frac{1}{1+r^2}(1-t)^{-\alpha} \right).
\end{split}\end{equation}
Moreover the inequality \eqref{gy:eq:straight line} is strict for all $t\in (0,1)$, unless $r=0.$
Hence we obtain
\begin{align*}
    &\int_0^1 \left( r^2 t^2 +(1-t)^2 \right)^{-\alpha/2} \dd t  \\ &\leq (1+r^2)^{-\alpha/2} \left( \frac{r^2}{1+r^2}  \int_0^1 t^{-\alpha} \dd t+ \frac{1}{1+r^2} \int_0^1 (1-t)^{-\alpha} \dd t \right) = \frac{(1+r^2)^{-\alpha/2}}{1-\alpha}.
\end{align*}
If there is equality, then there must be equality in \eqref{gy:eq:straight line} for almost every $t\in (0,1)$, and thus $r=0.$

\end{proof}

\begin{customproof}{Theorem \ref{thm:intro:alpha and n is 1}}
Let $\Gamma_1$ be as in \eqref{eq:quadrant_lemma_alpha_neg} and let $\gamma:[0, L] \to \R^2$ be the arc-length parametrization of $\Gamma_1$ with $L=|\Gamma_1|$ and
counterclockwise orientation. Without loss of generality, we can restrict our attention to the case where $\gamma(0)=(r,0)$ for some $r \in (0,1)$ and $\gamma(L)=(0,1)$, as any other case can be obtained from this by scaling and if necessary also interchanging the axes (i.e., taking the reflection of $\Gamma_1$ with respect to the diagonal). Since $\Gamma_1$ is the graph of a concave and decreasing function, $\gamma$ is Lipschitz, $|\gamma'|=1$, and both $\gamma_1'$ and $\gamma_2'$ are decreasing.

Therefore we obtain for $i=1,2$ and every $s \in (0,L)$ that
    \begin{align*}
        \gamma_i(s)-\gamma_i(0)&=\int_0^s \gamma_i'(t) \dd t =s \fint_0^s \gamma_i'(t) \dd t \\
        &\geq s \fint_0^L \gamma_i'(t) \dd t = \frac{s}{L} \left( \gamma_i(L)-\gamma_i(0)\right)
    \end{align*}
   where we used the fact that if a function is decreasing, its integral mean is also decreasing. This gives that
\begin{equation} \label{prop:geom_est_on_gamma}
\gamma_i(s) \geq \left( 1- \frac{s}{L}\right)\gamma_i(0) + \frac{s}{L} \gamma_i(L)\quad\text{ for }i=1,2.
\end{equation}
Consequently, by using the initial and endpoint values of $\gamma$, we obtain
\begin{equation}
\label{eq:gamma norm Prosenjit estim}
|\gamma(s)|^2 \geq r^2\left(1-\frac{s}{L} \right)^2 +\left(\frac{s}{L}\right)^2.
\end{equation}

Thus, we can estimate the weighted perimeter as 
\begin{align*}
     \int_{\Gamma_1} |x|^{-\alpha} \dHo(x) &=\int_0^{L} |\gamma(s)|^{-\alpha} \dd s \\
     &\leq  \int_0^{L} \left[ r^2\left(1-\frac{s}{L} \right)^2 +\left(\frac{s}{L}\right)^2\right]^{-\alpha/2} \dd s= L\int_0^1 \left[ r^2t^2 +(1-t)^2 \right]^{-\alpha/2} \dd t
\end{align*}
where we performed the change of variable $s=L(1-t)$.

Thus, we can establish \eqref{eq:quadrant_lemma_alpha_neg} by proving
\begin{equation} \label{eq:ineq_reduced}
\int_0^1 \left[ r^2t^2 +(1-t)^2 \right]^{-\alpha/2} \dd t \leq \frac{1}{1-\alpha}L^{-\alpha}.
\end{equation}

Thanks to the monotonicity of the perimeter of convex sets with respect to inclusion, we must have that $L\leq 1+r$. Here, we apply this monotonicity to a set $A$ which is in all four quadrants (obtained by extending $\Gamma_1$ via reflections across the canonical axes) which is contained in some rectangle $B$.
Thus the aforementioned inequality \eqref{eq:ineq_reduced} follows if we can show that
$$\int_0^1 \left[ r^2t^2 +(1-t)^2 \right]^{-\alpha/2} \dd t \leq \frac{(1+r)^{-\alpha}}{1-\alpha} \quad \text{for }r \in (0,1).$$
This last inequality is proven in the next two lemmas, namely  Lemma \ref{lm:integral_ineq_alpha} and Lemma \ref{lm:pointwise_ineq_}, concluding the proof of \eqref{eq:quadrant_lemma_alpha_neg}.

If there is equality in the theorem, there must be equality in all the above steps, in particular there must be equality in \eqref{eq:gamma norm Prosenjit estim} for almost every $s$ in the domain of $\gamma.$ Arguing by contradiction, this in turn yields that there must be equality in \eqref{prop:geom_est_on_gamma} for almost every $s$, and thus $\gamma_1$ and $\gamma_2$ are both linear functions and $\Gamma_1$ is a straight line. Now $\Gamma_1$ has to be equal to $\{0\} \times [0,L]$, otherwise we get a contradiction to Lemma \ref{gy:lemma:quadrant straight line}. 
\end{customproof}

\begin{Lemma} \label{lm:integral_ineq_alpha}
    Let $r \in (0,1)$ and $\alpha \in (0,1)$. Then the following inequality holds
    \begin{equation} \label{eq:integral_ineq_alpha}
        \int_0^1 \left(r^2  t^2 + (1-t)^2\right)^{-\alpha/2} \dd t \leq \frac{(1+r)^{-\alpha}}{1-\alpha}.
    \end{equation}
\end{Lemma}
\begin{proof}
For a given $\alpha \in (0,1)$ define the function
    \[
    F(r):= (1+r)^{\alpha}\int_0^1 \left(r^2  t^2 + (1-t)^2\right)^{-\alpha/2} \dd t.
    \]
    Proving \eqref{eq:integral_ineq_alpha} is equivalent to showing that 
    $F(r)-F(0) \leq 0$ for all $r \in (0,1)$. Since $F$ is differentiable, we will show this by proving that $F'(r) \leq 0$ for $r \in (0,1)$.
    Computing $F'(r)$ we get the following expression
    \begin{align*}
        F'(r) = \alpha(1+r)^{\alpha-1} \left[ \int_0^1 \left((1-t)^2-r t^2 \right)\left(r^2t^2  + (1-t)^2\right)^{-1-\alpha/2} \dd t\right].
    \end{align*}
    In order to prove that $F'$ is decreasing it is enough to show that
    \[
    G(r):=\int_0^1 \left((1-t)^2-r t^2 \right)\left(r^2t^2  + (1-t)^2\right)^{-1-\alpha/2} \dd t \leq 0.
    \]
    By performing the change of variables $t=\frac{1}{1+ru}$ we obtain
    \begin{align*}
        G(r)&=\int_0^\infty \left(\frac{r(ru^2-1)}{(1+ru)^2}\right)\left(\frac{r^2(1+u^2)}{(1+ru)^2}\right)^{-1-\alpha/2} \frac{r}{(1+ru)^2} \dd u \\
        &= r^{-\alpha} \int_0^\infty (1+ru)^{\alpha-2} (1+u^2)^{-1-\alpha/2} (ru^2-1) \dd u.
    \end{align*}
    We now split the domain of integration $(0,\infty)$ into $\left( 0,\frac{1}{\sqrt{r}} \right)$ and $\left(\frac{1}{\sqrt{r}},\infty \right)$
    \begin{align*}
        G(r)=  r^{-\alpha} &\left[ \int_0^{\frac{1}{\sqrt{r}}} (1+ru)^{\alpha-2} (1+u^2)^{-1-\alpha/2} (ru^2-1) \dd u + \right. \\
        & \left. + \int_{\frac{1}{\sqrt{r}}}^{\infty} (1+ru)^{\alpha-2} (1+u^2)^{-1-\alpha/2} (ru^2-1) \dd u \right] =:r^{-\alpha} \left[ I_1 + I_2 \right].
    \end{align*}
    Note that the integrand of $I_1$ is negative in its domain of integration $u\in(0,1/\sqrt{r}),$ whereas the integrand of $I_2$ is positive for $u\in (1/\sqrt{r},\infty).$
In order to compare $I_1$ and $I_2$, we let $v=\frac{1}{ru}$ in $I_2$ which gives
\begin{align*}
    I_2&=\int_{\frac{1}{\sqrt{r}}}^{\infty} (1+ru)^{\alpha-2} (1+u^2)^{-1-\alpha/2} (ru^2-1) \dd u \\
    &= \int_0^{\frac{1}{\sqrt{r}}} \left(1+\frac{1}{v}\right)^{\alpha-2} \left(1+\frac{1}{r^2v^2}\right)^{-1-\alpha/2} \left(\frac{1}{r v^2}-1\right) \frac{\dd v}{rv^2} \\
    &= r^{\alpha} \int_0^{\frac{1}{\sqrt{r}}}  (1+v)^{\alpha-2} (1+r^2v^2)^{-1-\alpha/2} (1-r v^2) \dd v. 
\end{align*}
Hence we can rewrite $G(r)$ as
\begin{align}
    G(r) = \int_0^{\frac{1}{\sqrt{r}}} (1-r v^2)\left[ (1+v)^{\alpha-2} (1+r^2v^2)^{-1-\alpha/2}  -r^{-\alpha} (1+rv)^{\alpha-2} (1+v^2)^{-1-\alpha/2}\right] \dd v.
\end{align}
This manipulation of $G(r)$ enables us to reduce nonpositivity of $G$ to a pointwise estimate of the integrand. Specifically, we will prove the following inequality
  \begin{equation} \label{pointwise_ineq_}
    r^{\alpha} \left( \frac{1+v}{1+rv} \right)^{\alpha-2} \leq \left( \frac{1+r^2v^2}{1+v^2} \right)^{1+\alpha/2} \quad \text{for } v \in \left(0, \frac{1}{\sqrt{r}}\right).
    \end{equation}
    This inequality is ensured through Lemma \ref{lm:pointwise_ineq_}, thus proving \eqref{eq:integral_ineq_alpha}.
\end{proof}

\begin{Lemma}\label{lm:pointwise_ineq_}
    Let $r \in (0,1)$ and $\alpha \in (0,1)$. Then
    \begin{equation} \label{pointwise_ineq_statement}
    r^{\alpha} \left( \frac{1+v}{1+rv} \right)^{\alpha-2} \leq \left( \frac{1+r^2v^2}{1+v^2} \right)^{1+\alpha/2} \quad \text{for } v \in \left(0, \frac{1}{\sqrt{r}}\right).
    \end{equation}
    \end{Lemma}
\begin{proof}
    By setting $s=\sqrt{r} v$ and $m=\sqrt{r}$, we can conveniently rewrite \eqref{pointwise_ineq_statement} as follows
    \[
   \left(\frac{m^2+s^2}{1+m^2s^2}\right)^{1+\alpha/2}=:X^{1+\alpha/2} \leq Y^{1-\alpha/2}:=\left( \frac{(m+s)^2}{(1+ms)^2}\right)^{1-\alpha/2}
   \quad \text{for all } m\in(0,1)\text{ and }s\in (0,1).
    \]

We claim that the quantities $X$ and $Y$ satisfy the condition $0 \leq X \leq Y \leq 1$.
    Clearly both $X,Y$ are nonnegative. We first check $Y\leq 1$.
    \begin{align*}
        Y\leq 1 \iff (m+s)^{2} - (1+ms)^2=(m^2-1)(1-s^2) \leq 0,
    \end{align*}
    which is clearly true as $m^2$ and $s^2$ are both smaller than $1.$
    It is also not difficult to see that $X \leq Y$, as shown by the following calculation:
    \begin{align*}
        Y-X &= \frac{(m+s)^2}{(1+ms)^2}- \frac{m^2+s^2}{1+m^2s^2} \\
        &=\frac{(m+s)^2(1+m^2s^2)- (m^2+s^2)(1+ms)^2}{(1+ms)^2(1+m^2s^2)} \\
        &=\frac{2ms(1-m^2)(1-s^2)}{(1+ms)^2(1+m^2s^2)} \geq 0.
    \end{align*}
 Using these conditions on $X$ and $Y$, and since $1+\alpha/2 \geq 1-\alpha/2$, we have
\[
X^{1+\alpha/2} \leq Y^{1+\alpha/2} \leq Y^{1-\alpha/2},
\]
which gives \eqref{pointwise_ineq_statement}.
\end{proof}

\begin{Remark}\label{remark:Almut}
Theorem \ref{thm:intro:alpha and n is 1} provides an alternative proof of  inequality \eqref{eq:log_inequality} in Corollary \ref{Cor:intro:p and log ineq} (without equality case) by taking the derivative with respect to $\alpha$ at $\alpha=0.$  Indeed, by Theorem \ref{thm:intro:alpha and n is 1} we have that
\begin{align*}
   \frac{ E_{-\alpha}(\Gamma) - E_0(\Gamma)}{\alpha}  \leq \frac{\frac{4^{\alpha}}{1-\alpha}E_0(\Gamma)^{1-\alpha}-E_0(\Gamma)}{\alpha} \quad \text{for all $\alpha \in (0,1)$},
\end{align*}
Thus, by taking $\alpha \to 0$ in the above expression one obtains \eqref{eq:log_inequality}.
\end{Remark}

\subsection{Minimizing $E_{2}$ in $\R^2$}
\label{section:p=2}

We devote this section to proving a result concerning the weight $|x|^2$ which has already been discovered by Sachs in  \cite{SachsIandII} and thus is not new. However,  the works of Sachs are difficult to access,  a bit sketchy as they sometimes rely on pictures, and only available in German.\footnote{Notes containing a partial English translation are available on the website of the first author, see reference \cite{SachsIandII} for a link.} Thus we will provide a proof, since we will use Theorem \ref{lm:quad_lemma_p=2} later in Section \ref{section:majorization} and to make our paper self-contained.  Most importantly, Theorem \ref{lm:quad_lemma_p=2} will serve as a fundamental tool to generalize the result to a broader class of weights (see Theorem \ref{Thm:general_phi}).

Our proof is completely different from the one by Sachs, arguably simpler, and mainly algebraic, but has the drawback of giving little geometric insight and intuition. We discovered this proof before finding the results of Sachs. It is inspired by the proof of the Riesz rearrangement inequality by \cite{hardy1952inequalities} via a discretization and induction argument. In our case we approximate the curve by a piecewise linear one and use the fact that on each piece the integrals of $|x|^2$ can be calculated explicitly. Sachs \cite{SachsIandII} uses an approximation  by polytopes, and an induction on the number of sides. 

\begin{Theorem} \label{lm:quad_lemma_p=2}
Let $\Gamma \subset \R^2$ be the graph of a concave decreasing function in the first quadrant connecting the coordinate axes. Then
\begin{equation} \label{ineq:conj2}
 \int_{\Gamma} |x|^2 \dHo(x) \geq \int_{0}^{|\Gamma|} t^2  \dd t =    \frac{|\Gamma|^3}{3}
\end{equation}
and equality holds if $\Gamma$ is a straight line segment. 
\end{Theorem}

Note that we do not deal with the case of equality, as we will not need it for our application in Section \ref{section:majorization}. We will comment on the equality case below after Corollary \ref{thm:p=2}. 

Theorem \ref{lm:quad_lemma_p=2} can also be stated as follows. If $f$ is a nonnegative, decreasing, concave function on $[0,L]$ such that $f(L)=0$, then
\begin{equation}\label{eq:gy:quadrant lem p is 2 with f}
  \int_0^L (t^2+f(t)^2)\sqrt{1+f'(t)^2}dt
  \geq\frac{1}{3}\left(\int_0^L\sqrt{1+f'(t)^2}dt\right)^3,
\end{equation}
with equality if $f(t)=m(L-t)$ for some $m\geq 0.$ The condition $f(L)=0$ is actually not necessary, as can be seen by a simple shifting argument.

This result readily implies the following global result.

\begin{Corollary} \label{thm:p=2}
Let $\Gamma \subset \R^2$ be a convex curve in $\R^2$, which has two orthogonal axes of symmetry meeting at the origin.  Then,
\begin{equation} \label{eq:conj_p=2}
{E_2(\Gamma)=}\int_{\Gamma} |x|^{2}  \dHo(x) \geq \frac{1}{48}|\Gamma|^{3}.
\end{equation}
Equality holds if $\Gamma$ is the \textit{needle}.
\end{Corollary}

It is not true that if equality holds in \eqref{eq:conj_p=2} then $\Gamma$ is a needle. By Theorem \ref{lm:quad_lemma_p=2} we have equality also for centrally symmetric parallelograms. In addition we miss the rectangles. Indeed  if $a+b=c$ is constant then for a quarter of the centered rectangle with sides $2a$ and $2b$ we get that
$$
  E_2(\Gamma)/4=\int_0^a (t^2+b^2)dt+\int_0^b(t^2+a^2)dt=(a+b)^3/3=c^3/3
$$
is constant, thus has the same energy as the needle with same length $E_0(\Gamma)$. With our method of proof of Theorem \ref{lm:quad_lemma_p=2} we miss the rectangles, as they are not the graph of a function. The problem is the approximation argument in our proof, as the rectangle is only the limit of graphs, but not a graph itself. 
Sachs \cite[2. Mitteilung, Satz 2]{SachsIandII} proved that these are actually the only cases when equality can occur.

We start by stating three inequalities that will be used repetitively. 

\begin{Lemma}[Three simple inequalities]
        \begin{equation}  \label{ineq2}
            b\sqrt{1+a^2} \geq a \sqrt{1+b^2} \qquad \text{for }b\geq a  \geq 0,
        \end{equation}
        \begin{equation}  \label{ineq3}
            1+b\sqrt{1+a^2} \geq \sqrt{1+a^2}\sqrt{1+b^2}\qquad \text{for }b\geq a  \geq 0,
        \end{equation}
        \begin{equation}  \label{ineq4}
           bc\sqrt{1+a^2} + \sqrt{1+c^2}\geq \sqrt{1+a^2}\sqrt{1+b^2}\sqrt{1+c^2} \qquad \text{for }c\geq b\geq a  \geq 0.
        \end{equation}

\end{Lemma}

\begin{proof}
Inequality \eqref{ineq2} is elementary and follows by taking the squares on both sides, for instance.

\eqref{ineq3} can be easily seen by taking the square on both sides, simplifying, and using eventually  that $2b\sqrt{1+a^2}\geq 2a\sqrt{1+a^2}\geq 2a^2\geq a^2.$

To prove \eqref{ineq4} define the function
\[
h(a,b,c):= bc \sqrt{1+a^2} +\sqrt{1+c^2}-\sqrt{1+a^2} \sqrt{1+b^2}\sqrt{1+c^2}
\]
Note that $h(a,b,b) = \sqrt{1+b^2}-\sqrt{1+a^2} \geq 0$ and
\begin{align*}
    \p_c h &= b\sqrt{1+a^2} + (1-\sqrt{1+a^2}\sqrt{1+b^2}) \frac{c}{\sqrt{1+c^2}}  \\
    &\geq b\sqrt{1+a^2} + (1-\sqrt{1+a^2}\sqrt{1+b^2}) \sup_{c}\frac{c}{\sqrt{1+c^2}} \\
    &= b\sqrt{1+a^2} + 1-\sqrt{1+a^2}\sqrt{1+b^2} \geq 0,
\end{align*}
where we have used in the last inequality \eqref{ineq3} to conclude.
\end{proof}

The proof of Theorem \ref{lm:quad_lemma_p=2} relies on an approximation by piecewise affine functions, reducing the statement to an algebraic inequality involving the slopes.

\begin{customproof}{Theorem \ref{lm:quad_lemma_p=2}}
We prove the theorem by approximating the function with a sequence of piecewise decreasing linear functions on equidistant intervals. More precisely, given $n\in\mathbb{N}$, in the formulation of \eqref{eq:gy:quadrant lem p is 2 with f}, we show the inequality for a function $f:[0,n]\to [0,\infty)$ such that $f$ is continuous and $f'(t)=a_i$ is constant if $t\in (i-1,i)$ for every $i=1,\ldots,n.$ Due to the convexity and nonnegativity assumption we must have that
$$
  0\leq a_1\leq a_2\leq\cdots\leq a_n\,.
$$
Due to the scaling invariance of the inequality to be proven, and by a simple approximation argument, it is sufficient to prove Theorem \ref{lm:quad_lemma_p=2} for such a function $f.$

Let us define $a=(a_1,\dots,a_n) \in \R^n$ and for $i = 1, \dots, n+1$ define
\[
A_i^{(n)} :=\sum_{k=i}^{n} a_k, \quad A_{n+1}^{(n)}=0,\quad\text{ so that }\quad f(i-1)=A_i^{(n)}.
\]
With these abbreviations we must have that
\begin{align}
    f(t)= A_i^{(n)} - a_i (t - i +1 ) \qquad \text{for }t\in (i-1, i).
\end{align}

Thus our goal is to prove  for every $n \in \N$ and for every $a=(a_1, \dots, a_n)$ with $a_n \geq \dots \geq a_1 \geq 0$ the following inequality holds
\begin{equation} \label{eq:quadlemmaaff}
      \int_0^n (t^2+f(t)^2)\sqrt{1+f'(t)^2}dt
  \geq\frac{1}{3}\left(\int_0^n\sqrt{1+f'(t)^2}dt\right)^3
  =\frac{1}{3} \left( \sum_{i=1}^{n} \sqrt{1+a_i^2} \right)^3.
\end{equation}
Let us compute the left hand side of the inequality in \eqref{eq:quadlemmaaff}, where we denote the graph of $f$ by $\Gamma$. Thus we obtain that
\begin{align*}
     &\int_{\Gamma} |x|^2 \dHo(x) = \sum_{i=1}^{n} \left(\int_{i-1}^{i} \left( t^2 + f(t)^2 \right) \sqrt{1+a_i^2} \dd t \right) \\
     &=\sum_{i=1}^{n}  \sqrt{1+a_i^2} \left( \frac{i^3-(i-1)^3}{3} +  \int_0^{1} (A_i^{(n)}-a_i t)^2 \dd t \right) \\
     &= \frac{1}{3} \sum_{i=1}^{n}  \sqrt{1+a_i^2} \left( i^3-(i-1)^3 - \frac{(A_i^{(n)}-a_i)^3 -  (A_i^{(n)})^3  }{a_i}\right)  \\
     &= \frac{1}{3} \sum_{i=1}^{n}  \sqrt{1+a_i^2} \left( i^3-(i-1)^3 + 3(A_i^{(n)})^2-3A_i^{(n)}a_i+ a_i^2\right).
\end{align*}
We now use that $A_i^{(n)}=A_{i+1}^{(n)}+a_i$ and $i^3-(i-1)^3=3i^2-3i+1$ to obtain
$$
  \int_{\Gamma} |x|^2 \dHo(x)=\frac{1}{3}
  \left[\sum_{i=1}^n(1+a_i^2)^{3/2}+3\sum_{i=1}^n\sqrt{1+a_i^2}
  \left(i^2-i+A_i^{(n)}A_{i+1}^{(n)}\right)
  \right].
$$
Thus \eqref{eq:quadlemmaaff} is equivalent to proving that
\begin{align}\label{eq:gy:E2 step 1}
\sum_{i=1}^{n} (1+a_i^2)^{3/2}+3 \sum_{i=1}^{n} \sqrt{1+a_i^2}\left(i^2-i + A_i^{(n)} A_{i+1}^{(n)}\right) \geq  \left(\sum_{i=1}^{n} \sqrt{1+a_i^2}\right)^3.
\end{align}
Let us assume for the moment that $n\geq 3$, introduce the notation
$$
  b_i:=\sqrt{1+a_i^2}\,,
$$
and observe that the right hand side of \eqref{eq:gy:E2 step 1} is equal to
$$
  \left(\sum_{i=1}^n b_i\right)^3=\sum_{i,j,k=1}^n b_i b_j b_k
  =\sum_{i=1}^n b_i^3+3\sum_{i=1}^n b_i\sum_{j=1,j\neq i}^n b_j^2
  +6 \sum_{1 \leq i <j<k\leq n} b_i b_j b_k\,.
$$
This yields that \eqref{eq:gy:E2 step 1} is equivalent to
$$
  \sum_{i=1}^{n} b_i\left(i^2-i + A_i^{(n)} A_{i+1}^{(n)}\right) \geq
  \sum_{i=1}^n b_i\sum_{j=1,j\neq i}^n b_j^2
  +2 \sum_{1 \leq i <j<k\leq n} b_i b_j b_k\,.
$$
Now use that
$$
  \sum_{i=1}^n b_i\sum_{j=1,j\neq i}^n b_j^2=
  \sum_{i=1}^n b_i\sum_{j=1,j\neq i}^n (1+a_j^2)=
  \sum_{i=1}^n b_i\left(n-1+\sum_{j=1,j\neq i}^n a_j^2\right)
$$
and the identity $i^2-i-(n-1)=(i-1)^2+i-n$. Inserting these identities we are able to
reduce the inequality \eqref{eq:gy:E2 step 1} to the nonnegativity of the auxiliary function $g_n(a)$ defined as
\begin{equation}\begin{split}
g_n(a)&:= \sum_{i=1}^{n} \sqrt{1+a_i^2} \left( (i-1)^2 + i - n + A_i^{(n)} A_{i+1}^{(n)} - \sum_{j=1, \dots, n, \, j \neq i } a_j^2 \right)  \\ &- 2 \sum_{1 \leq i <j<k\leq n} \sqrt{1+a_i^2} \sqrt{1+a_j^2} \sqrt{1+a_k^2}.
\label{eq:gy:gn defin}
\end{split}
\end{equation}
If $n=2,$ then the last term in the definition of $g_n$, the one involving the sum with the three indices $i<j<k$ has to be understood as $0$. With this convention the equivalence of \eqref{eq:gy:E2 step 1} with the nonnegativity of $g_n$ is valid also for $n=2,$ as can be easily checked by a simple computation.

Hence, inequality \eqref{eq:quadlemmaaff} is equivalent to proving that
\begin{align}
    g_n (a) \geq 0 \quad \text{for every $a=(a_1, \dots, a_n)$ with $a_n \geq \dots \geq a_1 \geq 0$}.
\end{align}
For the case $n=2$ we get
\begin{align*}
    g_2(a)&= \sqrt{1+a_1^2} \left( a_1a_2-1\right) + \sqrt{1+a_2^2}\left( 1-a_1^2\right) \\
    &= \left( \sqrt{1+a_2^2}- \sqrt{1+a_1^2}\right) + a_1 \left(a_2\sqrt{1+a_1^2} -a_1\sqrt{1+a_2^2}    \right) \\ &\geq a_1 \left(a_2\sqrt{1+a_1^2} -a_1\sqrt{1+a_2^2}    \right) \geq 0,
\end{align*}
where in the last inequality we used \eqref{ineq2}.

Now assuming the induction hypothesis we prove the general case. Note that
\[
A_i^{(n)} A_{i+1}^{(n)} = A_i^{(n-1)} A_{i+1}^{(n-1)} + a_n\left(  \sum_{k=i}^{n-1} a_k +\sum_{l=i+1}^{n-1} a_l + a_n \right).
\]
Exploiting this we can write again $g_n(a)$ putting in evidence $g_{n-1}(a)$ as
\begin{align*}
    g_n(a) &= \sum_{i=1}^{n-1} \sqrt{1+a_i^2} \left(  (i-1)^2 +i - (n-1) + A_i^{(n-1)} A_{i+1}^{(n-1)} - \sum_{j=1,\dots,n-1, \,j\neq i} a_j^2 \right) -\sum_{i=1}^{n-1}\sqrt{1+a_i^2} \\
    &+ \sum_{i=1}^{n-1} \sqrt{1+a_i^2} a_n \left( \sum_{k=i}^{n-1}a_k + \sum_{l=i+1}^{n-1} a_l \right) + \sqrt{1+a_n^2} \left((n-1)^2 + \sum_{j=1}^{n-1}a_j^2 \right) \\
    & - 2 \sum_{1 \leq i <j<k \leq n-1} \sqrt{1+a_i^2} \sqrt{1+a_j^2} \sqrt{1+a_k^2} - 2 \sum_{1 \leq i <j\leq n-1} \sqrt{1+a_i^2} \sqrt{1+a_j^2}  \sqrt{1+a_n^2} \\
    &= g_{n-1}(a)  -\sum_{i=1}^{n-1}\sqrt{1+a_i^2} +   \sqrt{1+a_n^2} \left((n-1)^2 - \sum_{j=1}^{n-1}a_j^2 \right)  \\
    &+  \sum_{i=1}^{n-1} \sqrt{1+a_i^2} a_n \left( \sum_{k=i}^{n-1}a_k + \sum_{l=i+1}^{n-1} a_l \right) - 2 \sum_{1 \leq i <j\leq n-1} \sqrt{1+a_i^2} \sqrt{1+a_j^2}  \sqrt{1+a_n^2}.
 \end{align*}

Now use that and that $\sqrt{1+a_i^2} \leq \sqrt{1+a_n^2}$ for $i=1,\dots,n$ and that $(n-1)^2-(n-1) = 2 \binom{n-1}{2}$, we obtain 
$$
   -\sum_{i=1}^{n-1}\sqrt{1+a_i^2} +   \sqrt{1+a_n^2} (n-1)^2\geq 2 \binom{n-1}{2} \sqrt{1+a_n^2}\,.
$$
Taking advantage of this inequality, of the induction hypothesis $g_{n-1}(a)\geq 0,$ and of the identity $\sum_{k=i}^{n-1}a_k=a_i+\sum_{l=i+1}^{n-1}a_l$ shows that
 \begin{align*}
     g_n(a) &\geq 2 \binom{n-1}{2} \sqrt{1+a_n^2} -  \sum_{i=1}^{n-1} a_i^2 \sqrt{1+a_n^2}   + \sum_{i=1}^{n-1}a_n a_i \sqrt{1+a_i^2}  \\
     &+ 2 \sum_{i=1}^{n-1} \sqrt{1+a_i^2} \, a_n \sum_{j=i+1}^{n-1} a_j  - 2   \sum_{1 \leq i <j\leq n-1} \sqrt{1+a_i^2} \sqrt{1+a_j^2}  \sqrt{1+a_n^2}\,.
\end{align*}
To estimate the second and third term we use \eqref{ineq2}, namely $a_n\sqrt{1+a_i^2}-a_i\sqrt{1+a_n^2}\geq 0,$   which implies that $g_n(a)$ is bounded below as 
\begin{align*}
      g_n(a) &\geq  2 \binom{n-1}{2} \sqrt{1+a_n^2} + 2 \sum_{1\leq i <j \leq n-1} \sqrt{1+a_i^2} \, a_j a_n   - 2   \sum_{1 \leq i <j\leq n-1} \sqrt{1+a_i^2} \sqrt{1+a_j^2}  \sqrt{1+a_n^2}\,.
 \end{align*}
 For each $j$ such that $i < j \leq n-1$ we know from \eqref{ineq4}  that
 \[
 \sqrt{1+a_n^2} + \sqrt{1+a_i^2} a_j a_n - \sqrt{1+a_i^2} \sqrt{1+a_j^2}  \sqrt{1+a_n^2} \geq 0,
 \]
which proves $g_n(a) \geq 0$.

\end{customproof}

\subsection{Results for general weights via a majorization argument}
\label{section:majorization}

In this section we prove Theorem \ref{Thm:general_phi} and Theorem \ref{thm:four quadrant convex}. The method heavily relies on Theorem \ref{lm:quad_lemma_p=2} of the previous section.

We first recall the definitions and some elementary properties of rearrangements, which is required in the proof of Theorem \ref{Thm:general_phi}.
    Let $f: [0, L] \to \mathbb{R}$ be a measurable, nonnegative function. The \textit{distribution function} of $f$ is defined as
    $$\mu_f(\tau) = |\{x \in [0, L] : f(x) > \tau\}|.$$
    The \textit{decreasing rearrangement} $f^{\star}$ is defined as
    $$f^{\star}(s) = \inf\{\tau \ge 0 : \mu_f(\tau) \le s\}.$$
    However, for us it will be more convenient to use the  \textit{increasing rearrangement} $f^*$, which is defined by reflecting the domain of the decreasing rearrangement, namely
    $$f^*(s) = f^{\star}(L - s).$$
    Equivalently, it could have been defined directly as $f^*(s) = \sup\{\tau \ge 0 : |\{x \in [0,L] : f(x) < \tau\}| \le s\}$. By construction, $f^*$ is an increasing function on $[0, L]$ and is equimeasurable with $f$.

We will use the following two known properties of rearrangements; see for instance \cite{Baernstein} for the first property,  and \cite[p. 81, Chapter 3]{lieb2001analysis} or \cite[p. 277]{hardy1952inequalities} for the second property.

$(i)$ Let $f: [0, L] \to \mathbb{R}$ be a measurable, nonnegative Lipschitz function with Lipschitz constant $[f]_{Lip([0,L])}$. Then
    \begin{equation} \label{lm:rearrange_is_lip}
    [f^*]_{Lip([0,L])}\leq  [f]_{Lip([0,L])}.
    \end{equation}

$(ii)$ For any nonnegative measurable function $f:[0,L] \to \R$ and absolutely continuous function $\psi:[0,\infty) \to \R$, we have that
    \begin{equation}\label{eq:invar integ under rearr}
       \int_0^L \psi(f(t)) \dd t = \int_0^L \psi(f^*(t)) \dd t.
    \end{equation}

The next lemma is a variant of classical majorization theorems, see  for instance, \cite[Theorem D.2 page 22, Proposition E.4 page 751]{Marshall}. However, \cite{Marshall} does not deal with the case of equality. We prove here a version tailored to our purpose. For a sharper version of a similar majorization result, we refer to \cite[Theorem 2.3]{Chong}.

\begin{Lemma} \label{prop:maj_tool} 
Suppose that $L>0$ and that  $a$ is a bounded increasing function on $[0,L]$ and $b\in L^{\infty}((0,L)).$  Assume that $I \subset \mathbb{R}$ is a bounded open interval and that $a([0,L]), b([0,L]) \subset \overline{I}$. 
Let $\phi:I  \to \R$ be a convex and decreasing function such that  $\phi \circ b \in L^1((0,L))$, and assume that the following majorization condition holds
\begin{equation} \label{eq:maj_assumption}
    \int_0^x a(t) \dd t \ge \int_0^x b(t) \dd t \quad \text{for all } x \in (0, L).
\end{equation}
Then,
\begin{equation} \label{eq:maj_conclusion}
    \int_0^L \phi(a(t)) \dd t \le \int_0^L \phi(b(t)) \dd t 
\end{equation}
and $\phi\circ a\in L^1((0,L))$.

Conversely, if $\phi$ is a concave and increasing function, the inequality in \eqref{eq:maj_conclusion} is reversed.

Moreover, if $\phi$ is strictly convex (or strictly concave) and there is equality in \eqref{eq:maj_conclusion}, then $a(t)=b(t)$ for almost every $t\in(0,L).$
\end{Lemma}

\begin{proof} We first focus on showing inequality \eqref{eq:maj_conclusion}.
Since $\phi$ is convex, its left derivative $\phi'_-$ is well defined, increasing, and is a subgradient at every point (any choice of a subgradient would work for this proof).

Let $-\infty<\alpha<\beta<\infty$ be such that $I=(\alpha,\beta).$ A convex decreasing function on $(\alpha,\beta)$, if not bounded, can blow up only at $\alpha$ such that $\lim_{x\searrow \alpha}\phi(x)=+\infty.$ If this is the case, we shall define a family of cutoff functions $\{\phi_m\}_{m\geq 1}$ such that $\phi$ is replaced by a linear function near $\alpha$ in the following way (if $\phi$ is bounded, just set $\phi_m=\phi$ everywhere in the proof below):
$$
  \phi_m(x):=\left\{
  \begin{array}{ll}
    \phi(x) & \text{ if }x> \alpha+1/m
    \\
    \smallskip
    \phi(\alpha+1/m)+\phi_-'(\alpha+1/m)(x-(\alpha+1/m)) & \text{ if }x\in (\alpha,\alpha+1/m].
  \end{array}\right. 
$$
Note that $\phi_m$ is well defined for $m$ big enough, and also a convex decreasing function. We can thus define
$$
  S_m(t):=\left(\phi_m\right)_-'(a(t))\quad\text{ for every }t\in[0,L].
$$
For every integer $m\geq 1$  the function $S_m$  is  increasing , since it is the composition of two increasing functions. Moreover it holds that $S_m$ is bounded and $S_m(t)\leq 0$ for every $t\in [0,L].$  

Since $\phi_m$ is convex, the first-order Taylor expansion dictates that
\begin{equation}
\label{eq:gy:a is b if phi strict}
\phi_m(b(t)) - \phi_m(a(t)) \ge S_m(t) (b(t) - a(t)) \quad \text{for  every } t \in (0, L).
\end{equation}
Integrating both sides over $[0, L]$ gives
\begin{equation} \label{eq:maj_proof_1}
    \int_0^L \big( \phi_m(b(t)) - \phi_m(a(t)) \big) \dd t \ge \int_0^L S_m(t) (b(t) - a(t)) \dd t.
\end{equation}

Let us define the absolutely continuous function $G(x) = \int_0^x (b(t) - a(t)) \dd t$. By the majorization assumption \eqref{eq:maj_assumption}, we know that $G(x) \le 0$ for all $x \in [0, L]$, and clearly $G(0) = 0$.

We will first prove that for every bounded increasing nonpositive function $S$ 
\begin{equation}\label{eq:SG partial integ}
  \int_0^L S(t)(b(t)-a(t))\dd t=\int_0^L S(t)G'(t)\dd t\geq 0\quad\text{ if $G\leq 0$ and $G(0)=0$}.
\end{equation}
This last inequality is easily verified by partial integration if $S\in C^1([0,L])$ in the following way. From the identity 
\begin{align*} 
\int_0^L S(t) G'(t) \dd t =& S(L)G(L) - S(0)G(0) - \int_0^L G(t) S'(t)dt\\
=&S(L)G(L)-\int_0^L G(t) S'(t)\dd t\,,
\end{align*}
the assumptions $S\leq 0$, $G\leq 0$, and $S'\geq 0$ we obtain the inequality \eqref{eq:SG partial integ} in this case.

For general $S$ we argue by approximation. Since $S$ is bounded and increasing we have that  $S\in L^1((0,L))$ and thus $S$ can be approximated by a sequence of functions $S_{\epsilon}\in C^1([0,L])$ such that $\lim_{\epsilon\to 0} \|S_{\epsilon}-S\|_{L^1}=0.$ Moreover, $S_{\epsilon}\leq 0$ and $S_{\epsilon}$ is increasing, since approximation by a standard convolution method with a nonnegative convolution kernel preserves these two properties. This concludes the proof of \eqref{eq:SG partial integ} for general $S.$ 

It now follows from  \eqref{eq:SG partial integ} applied to $S=S_m$ that
\begin{equation}\label{eq:SG applied to Sm}
  \int_0^L S_m(t) (b(t) - a(t)) \dd t\geq 0. 
\end{equation}
This last inequality combined with \eqref{eq:maj_proof_1}  shows that  for every $m\geq1$
$$
  \int_0^L\phi_m(a(t))\dd t\leq \int_0^L\phi_m(b(t))\dd t.
$$
It now follows from letting $m\to\infty$ and  monotone convergence theorem that \eqref{eq:maj_conclusion} holds.  Note that the monotone convergence theorem applies, even though the $\phi_m$ could change sign, since the $\phi_m$ are bounded by below by the infimum of $\phi$ on the bounded interval $I.$

Finally, to show that $\phi\circ a$ is in $L^1$ we decompose $\phi$ into its positive and negative part as $\phi=\phi^+-\phi^-$ and observe that $\phi^-$ must be bounded by the assumptions on $\phi$ and by the boundedness of the interval $I.$ Hence we obtain from \eqref{eq:maj_conclusion} and from the hypothesis $\phi\circ b\in L^1$ that $\phi^+\circ a$ belongs to $L^1$ too.

 The case where $\phi$ is concave and increasing follows  by applying the lemma to $-\phi$. 

\smallskip 

\textit{\underline{Proof of the Case of Equality.}} 
Let us define two functions on $(0,L)$ by 
\begin{align*}
    H_m(t):=& \phi_m(b(t)) - \phi_m(a(t)) - S_m(t) (b(t) - a(t))\qquad\text{and}
    \\
    H(t):=&\phi(b(t)) - \phi(a(t)) - \phi_-'(a(t)) (b(t) - a(t)).
\end{align*}
If $\phi$ is not bounded, then it follows from $\phi\circ a$, $\phi\circ b\in L^1$ that  $a(t)$ and $b(t)$ are strictly greater than $\alpha$ for almost every $t\in(0,L)$. 
Thus it follows that $H_m(t)$ converges to $H(t)$ for almost every $t\in (0,L]$ since $\phi_m(x)$ converges to $\phi(x)$ for every $x\in (\alpha,\beta)$, and  $S_m$ converges to $\phi_-'$ pointwise too. 

Note that \eqref{eq:gy:a is b if phi strict} and \eqref{eq:maj_proof_1} hold also with $\phi_m$ replaced by $\phi,$ respectively $S_m$ replaced by $\phi_-'$, as \eqref{eq:gy:a is b if phi strict} uses merely the convexity of $\phi_m$, but not the fact that $\phi_m$ has bounded derivative. In particular $\int_0^L H(t)\dd t\geq 0$.
On the other hand, from \eqref{eq:SG applied to Sm} we conclude that
$$
  \int_0^L H_m(t)\dd t\leq \int_0^L \left(\phi_m(b(t))-\phi_m(a(t))\right)\dd t.
$$
From this last inequality it follows by Fatou's lemma, monotone convergence theorem for $\phi_m\circ a$ and $\phi_m\circ b$, and the assumption of equality in  \eqref{eq:maj_conclusion} that
\begin{align*}
    0 \leq \int_0^L H(t)\dd t \leq \liminf_{m\to\infty}\int_0^L \left(\phi_m(b(t))-\phi_m(a(t))\right)\dd t=\int_0^L\phi(b(t))\dd t-\int_0^L\phi(a(t))\dd t=0. 
\end{align*}
This shows that $H(t)=0$ for almost every $t\in (0,L).$ Using the strict convexity of $\phi$ this yields that $a(t)=b(t)$ for almost every $t\in (0,L).$ 
\end{proof}

We have now all the tools needed for the proof of Theorem \ref{Thm:general_phi}.

\begin{customproof}{Theorem \ref{Thm:general_phi}}
Let $\Gamma_1=\Gamma\cap \{(x,y)\in\R^2;\, x\geq 0,\,y\geq 0\}$ be the intersection of $\Gamma$ with the first quadrant. We will first prove an estimate in the first quadrant, which will immediately extend to the other three quadrants due to the symmetry assumption on $\Gamma.$

    Let us denote $L=|\Gamma_1|=|\Gamma|/4$ and $\gamma$ be a parametrization by arc-length of $\Gamma_1$ with parameter $t \in [0, L]$.
Let $r(t) = |\gamma(t)|$. Because $\gamma$ is parameterized by arc-length, $|\gamma'(t)| = 1$.
This yields
$$ |r(t) - r(s)| = ||\gamma(t)| - |\gamma(s)|| \le |\gamma(t) - \gamma(s)| \le |t - s|. $$
Thus, $r$ is a Lipschitz function with $[r]_{Lip[0,|\Gamma|]}\leq 1$ and by \eqref{lm:rearrange_is_lip} also $r^*$ enjoys the same property.
From \eqref{eq:invar integ under rearr} and Theorem \ref{lm:quad_lemma_p=2}, we know that (Theorem \ref{lm:quad_lemma_p=2} is only stated for graphs, but it is valid also if $\Gamma_1$ is not a graph, by approximation)
\begin{equation} \label{eq:quad_lemma_p=2}
    \int_0^{L} r^*(s)^2 \dd s=\int_0^{|\Gamma_1|} r(s)^2 \dd s
    =\int_{\Gamma_1}|x|^2 d\mathcal{H}^1(x)\geq \frac{|\Gamma_1|^3}{3}
    =\int_0^{L}t^2 dt.
\end{equation}
Note that we only use the convexity in this step. 
We now claim that
\begin{equation} \label{eq:maj_hp_p=2}
\int_0^x r^*(t)^2 \, dt \ge \int_0^x t^2 \, dt \quad \text{for all }x \in [0,L],
\end{equation}
which holds true for any nonnegative increasing $1$-Lipschitz function $r^*$ satisfying \eqref{eq:quad_lemma_p=2}. To show this claim we define
$$
  G(x)=\int_0^x \left(r^*(t)^2-t^2\right)dt,
$$
which satisfies $G(0)=0,$ $G(L)\geq 0$. Our goal is to  prove that $G$ is nonnegative for all $x\in(0,L)$. Let us assume that $G$ achieves its minimum for some $x_0\in (0,L)$. Thus we obtain that $G'(x_0)=0$ which in turn gives $r^*(x_0)=x_0$ as $r^*$ is nonnegative. Using $r^*$ is $1$-Lipschitz and increasing, we obtain
$$
  x_0-r^*(t)=r^*(x_0)-r^*(t)=|r^*(x_0)-r^*(t)|\leq x_0-t\qquad\forall\,t\in (0,x_0).
$$
Thus $r^*(t)\geq t$ for all $t\in (0,x_0),$ and this proves \eqref{eq:maj_hp_p=2} for $x=x_0.$ But, as $G(x_0)$ is the minimum of $G,$ we have proven \eqref{eq:maj_hp_p=2}.

To prove Theorem \ref{Thm:general_phi} we now apply Lemma \ref{prop:maj_tool} with 
$a(t)=r^*(t)^2$ and $b(t)=t^2$ and $\phi=F.$ 
The assumption \eqref{eq:maj_assumption} of the lemma is ensured by \eqref{eq:maj_hp_p=2}. Thus we obtain, using also \eqref{eq:invar integ under rearr} that
$$
  \int_{\Gamma_1}F(|x|^2)d\mathcal{H}^1(x)=\int_0^L F\big((r^*(t))^2\big)\dd t\leq 
  \int_0^L F\big(t^2\big)\dd t
$$
which proves the inequality in Part (I) of  the theorem. 

In case there is equality in Theorem \ref{Thm:general_phi}, then it follows from Lemma \ref{prop:maj_tool} that $r^*(t)^2=t^2$ for every $t.$ This yields that $\min(r)=\min(r^*)=0.$ A convex curve which has two orthogonal axes of symmetry meeting at the origin and that passes through the origin has to be a needle.

Part (II) of the theorem follows from Part (I) applied to $-F.$ 

We now show that one can remove the assumption that the two orthogonal axes of symmetry meet at the origin if $t \mapsto F(t^2)$ is  concave in Part (I) (the analogous result for Part (II) follows by applying the same argument to $-F$). Consider the case of a weight function as in Part (I) with $t \mapsto F(t^2)$ being concave and let $\Gamma_P$ be a convex curve with two orthogonal axes of symmetry meeting at the point $P \in \R^2$. Denote by $\Gamma_0$ the same curve shifted in such a way that the two axes of symmetry now meet at the origin. By showing that $E_{F(|\cdot|^2)}(\Gamma_P) \leq E_{F(|\cdot|^2)}(\Gamma_0)$ we establish the desired claim. Observe that 
\[
E_{F(|\cdot|^2)}(\Gamma_P) =\int_{\Gamma_P} F(|x|^2) \dHo(x) =\int_{\Gamma_0} F(|x-P|^2) \dHo(x). 
\]
Since for $x \in \Gamma_0$, by symmetry, we have also $-x \in \Gamma_0$, we can further express $E_{F(|\cdot|^2)}(\Gamma_P) $ as 
\begin{align*}
E_{F(|\cdot|^2)}(\Gamma_P) &= \frac{1}{2}\int_{\Gamma_0} \left(F(|x-P|^2) + F(|x+P|^2) \right) \dHo(x) \\
&\leq \int_{\Gamma_0} F \left( \left( \frac{|x-P| + |x+P|}{2} \right)^2 \right) \dHo(x) \leq \int_{\Gamma_0} F(|x|^2) \dHo(x)=E_{F(|\cdot|^2)}(\Gamma_0),
\end{align*}
where in the first inequality we used the concavity of $t \mapsto F(t^2)$ and in the second inequality that $F$ is decreasing.

To deal with the case of equality, note that due to the first part of the theorem $\Gamma_0$ must be a needle. If $\Gamma_0$ is a needle, then $E_{F(|\cdot|^2)}(\Gamma_P)= E_{F(|\cdot|^2)}(\Gamma_0)$ can only hold if $P=0,$ since $F$ is strictly decreasing (due to stirct convexity).

\end{customproof}

We now prove Theorem \ref{thm:four quadrant convex}, which is an easy consequence of Theorem \ref{Thm:general_phi}.

\begin{customproof}{Theorem \ref{thm:four quadrant convex}}
We first prove statement (I) of Theorem \ref{Thm:general_phi} for four-quadrant-convex curves. For $i=1,\ldots,4$ let $\Gamma_i$, $O_i$ and $S_i$ be as in Definition \ref{def:4 quadrant convex}. Thus by definition and thanks to Theorem \ref{Thm:general_phi} we have that
$$
  4 E_{F(|\cdot|^2)}(\Gamma_i) =E_{F(|\cdot|^2)}(S_i) 
  \leq 4\int_0^{\frac{|S_i|}{4}}F(t^2)dt=4\int_0^{|\Gamma_i|}F(t^2)dt
  \quad\text{ for $i=1,\ldots,4$.}
$$
Taking the sum of these four inequalities we obtain that 
$$
  E_{F(|\cdot|^2)}(\Gamma)\leq \sum_{i=1}^4\int_0^{|\Gamma_i|}F(t^2)dt.
$$
Now observe that $f(t):=F(t^2)$ is decreasing and thus the function $x\mapsto \int_0^x f(t)dt$ is concave, which yields that
\begin{equation}\label{eq:gy:eq in proof 4 quadr}
  \sum_{i=1}^4\int_0^{|\Gamma_i|}F(t^2)dt\leq 4 \int_0^{\frac{|\Gamma_1|+|\Gamma_2|+|\Gamma_3|+|\Gamma_4|}{4}}F(t^2)dt=
  4\int_0^{\frac{|\Gamma|}{4}}F(t^2)dt.
\end{equation}
This proves the desired inequality of Theorem \ref{thm:four quadrant convex} in the case (I).

Let us now deal with the case of equality. If there is equality in \eqref{eq:gy:eq in proof 4 quadr} and $F$ is strictly convex (and thus $f$ strictly decreasing), then all four $|\Gamma_i|$ must be equal. This proves the claim on the case of equality in Theorem \ref{thm:four quadrant convex}, making use of the corresponding statement in Theorem \ref{Thm:general_phi}, that is, each $S_i$ must be a needle. 

The proof of  (II) is just applying Part (I) of the theorem to $-F.$ 
\end{customproof}

\section*{Acknowledgements}
The authors are grateful to Guido De Philippis for inspiring the shifting argument which led to Theorem \ref{thm:shifting_generalized}. We also thanks Charlotte Dietze for pointing out the works \cite{Charlotte2024,Hurwitz1902} and Giacomo Di Paolo for helping us recover a copy of the manuscripts of Sachs \cite{SachsIandII,SachsIII,SachsIV,SachsV}.
We should also thank Almut Burchard for Remark \ref{remark:Almut}. We are very grateful to Richard Laugesen and Pedro Freitas for helping us find several relevant references.

\bibliographystyle{abbrv}
\bibliography{Bib}
\end{document}